\documentclass[11pt]{amsart}

\usepackage[utf8]{inputenc}
\usepackage[T1]{fontenc}
\usepackage{lmodern}
\usepackage{microtype}
\usepackage{amssymb}
\usepackage[letterpaper,margin=1in]{geometry}
\usepackage{enumitem}
\usepackage{xcolor}
\usepackage{graphicx}

\usepackage{natbib}
\usepackage[colorlinks,
            linkcolor=blue,
            citecolor=blue,
            urlcolor=blue]{hyperref}
\usepackage{cleveref}          

\hypersetup{
  pdftitle={Strong approximations for dependent random fields},
  pdfauthor={Soham Bonnerjee and Sayar Karmakar}
}

\numberwithin{equation}{section}

\theoremstyle{plain}
\newtheorem{theorem}{Theorem}[section]
\newtheorem{proposition}[theorem]{Proposition}
\newtheorem{lemma}[theorem]{Lemma}

\theoremstyle{definition}

\theoremstyle{remark}

\newtheorem*{example}{Example}

\newcommand{\R}{\mathbb{R}}
\newcommand{\C}{\mathbb{C}}

\newcommand{\Z}{\mathbb{Z}}
\newcommand{\N}{\mathbb{N}}
\newcommand{\E}{\mathbb{E}}
\newcommand{\IE}{\mathbb{E}}
\newcommand{\IP}{\mathbb{P}}

\newcommand{\as}{\mathrm{a.s.}}
\newcommand{\eps}{\varepsilon}
\newcommand{\ind}{\mathbf 1}

\newcommand{\bG}{\mathbf{G}}

\newcommand{\BX}{\mathbf{X}}
\newcommand{\BY}{\mathbf{Y}}

\newcommand{\Fcal}{\mathcal{F}}

\newcommand{\Lcal}{\mathcal{L}}

\newcommand{\cG}{\mathcal G}

\newcommand{\cQ}{\mathcal Q}
\newcommand{\cB}{\mathcal B}
\newcommand{\cD}{\mathcal D}

\newcommand{\norm}[1]{\Vert #1\Vert}
\newcommand{\inner}[1]{\langle #1\rangle}
\newcommand{\abs}[1]{\lvert #1 \rvert}

\newcommand{\cov}{\text{Cov}}

\newcommand{\wh}{\widehat}

\newcommand{\Cov}{\operatorname{Cov}}
\newcommand{\Var}{\operatorname{Var}}
\newcommand{\dist}{\operatorname{dist}}

\newcommand{\HA}{\mathbf{A}}   
\newcommand{\HO}{\mathbf{H}}   
\newcommand{\HD}{\mathbf{D}}   
\newcommand{\iu}{\mathrm{i}}   

\newcommand{\ignore}[1]{}

\definecolor{babypink}{rgb}{1.0, 0.25, 0.5}

\begin{document}

\title{Strong approximations for dependent random fields}

\author{Soham Bonnerjee}
\address{Department of Statistics and Data Science, The Wharton School, University of Pennsylvania.}
\email{sohamb@wharton.upenn.edu}

\author{Sayar Karmakar}
\address{Department of Statistics, University of Florida.}
\email{sayarkarmakar@ufl.edu}

\date{}


\begin{abstract}
We establish strong Gaussian approximations for centered stationary random fields on $\mathbb Z^d$, $d\ge2$, represented as measurable functions of i.i.d.\ innovations. Under finite $p$th moments, $p>2$, and sufficiently fast decay of spatial dependence, we construct a single coupling of the field with an i.i.d.\ Gaussian field. Its partial sums, scaled by the long-run standard deviation, approximate those of the original field almost surely, uniformly over all rectangular prefixes and all observation scales. Over a cube of side length $n$, the approximation error is $O_{\mathrm{a.s.}}(n^{\max\{d/p,(d-1)/2\}})$ up to logarithmic factors. While $d/p$ is the familiar moment term, the surface term $(d-1)/2$, absent for independent planar arrays, emerges as a cost of dependence, and cannot be removed even for finite-range Gaussian fields. The assumptions accommodate nonlinear and noncausal fields without imposing mixing. The proof combines a permanent dyadic blocking, a local Gaussian coupling in tensor Haar coordinates driven by a Doob martingale with vanishing conditional third moments, and a comparison of the resulting dependent Gaussian blocks with white noise.
\end{abstract}

\maketitle

\section{Introduction}
We consider the problem of strong Gaussian approximation for dependent spatial random fields. Concretely, 
for a centered stationary random field with finite $p$ moments, $(X_i)_{i\in\mathbb Z^d}$, we prove that there exists (ignoring measure theoretic nuances) an i.i.d Gaussian field $(Y_i)_{i\in\mathbb Z^d}$ and a constant $\sigma_{\infty}$ such that 
\begin{equation*}
\max_{m\le n}\abs{S_X(m)-\sigma_\infty S_Y(m)}
=O_{\as}(n^{\max\{ d/p, (d-1)/2 \} } \text{polylog $(n)$ terms}),
\end{equation*}
where, $S_X(m)=\sum_{ i \in A} X_i$, $A={\prod_{j=1}^d\{1,\ldots,m_j\}}$ with $m_i \leq n$ denotes the sum of the random variables $X$ over lower-left rectangular prefixes with top-right corner at $m\in \N^d$, and $S_Y(m)$ denotes the corresponding sum over the coupled Gaussian process $Y$. Strong approximation for the time-series case, i.e. $d=1$, has a long history, starting from the \citet{strassen1964invariance, kmt,kmt1976, csorgo1981strong} for i.i.d. observations, with \citep{philipp1975almost,berkes1979approximation, wu2005nonlinear,wu2007strong,liu2009strong, berkes2014kmt, wuzhou2011gaussian, karmakar2020optimal, bonnerjee2024gaussian} provide corresponding versions in time-series. However, passing from $d=1$ to $d>1$ comes with significant challenges. The main bottleneck can be summarized into the truism that $\Z^d$ is only partially ordered.
Moreover, control along a growing sequence of cubes does not by itself
control every thin rectangle inside them. Finally, the rates for $d=1$ do not provide intuition for the term $(d-1)/2$; for $d=2$ and independent random fields, \cite{rio1997} derived the rates $2/p$. However, we show that the rate $\max\{ d/p, (d-1)/2 \}$ is sharp for dependent random fields; specifically, the term $(d-1)/2$ arises as a \textit{cost of dependence}.

For independent arrays, \citet{morrow1986invariance} established
set-indexed strong approximations, while \citet{rio1993a,rio1993b, rio1997}
showed how rates depend on the indexing class. For dependent fields, strong approximations have largely been developed under mixing or association-type conditions \citep{berkes1981mixing,berkes1984gaussian,strittmatter1990strong,balan2005strong,bulinski2006strong,shashkin2007strong,cai2011strong}, while central limit theorems and weak invariance principles are available under mixing, projective, finite-dependence and orthomartingale conditions \citep{bolthausen1982central,dedecker1998central,dedecker2001exponential,wang2013new,volny2014hannan,giraudo2018invariance,cuny2026weak}. Our setting departs from this literature in three respects. First, we measure dependence through the functional-dependence coefficients of \citet{wu2005nonlinear}, which can be computed explicitly, as in \S\ref{ssc:examples}. Mixing and association, in contrast, are properties of the induced $\sigma$-fields or covariances, which are typically hard to verify, and strong mixing can fail even for simple linear processes. Moreover, functional-dependence and spatial Hannan conditions are not interchangeable \citep{volny2017erratum}, so results under projective criteria do not automatically cover our setting. Second, we impose no growth restriction on the rectangles, whereas, for instance, \citet{berkes1981mixing} require growth regions that stay away from the coordinate axes. Third, the closest work in terms of usage of functional dependence measures - \citet{el2013central}, establishes moment inequalities and central limit theorems for functionals of i.i.d.\ spatial innovations, but weak limit theorems do not deliver the pathwise rates of a coupling. 
To the best of our knowledge, obtaining a single, sharp, almost sure Gaussian
white-noise coupling for dependent random fields under general enough conditions  over unrestricted rectangles, has remained an open problem.
We settle this problem with an optimal approximation error.

Our main result, Theorem~\ref{thm:main} only assumes sufficiently fast decay of spatial dependence along with positive long-run variance.
It produces a coupling with an i.i.d Gaussian random field at all observation
scales, allowing nonlinear and noncausal fields without requiring strong
mixing. The construction uses a permanent dyadic partition to obtain independent
surrogate blocks with the correct marginal laws. The spatial versions of Doob or Burkholder-type maximal inequalities from
\citet{el2013central} and \citet{moricz1977moment} control this replacement
error. Within each block, tensor Haar coordinates reduce the simultaneous control of all lower-left rectangular sums to an $\ell^\infty$ bound on the Haar coefficients, at a logarithmic price. Revealing the innovations one binary digit at a time writes these coefficients as a martingale with vanishing conditional third moments, so that the third-order Yurinskii coupling of \citet{cattaneo2025yurinskii} applies without any symmetry assumption on the observations; this local coupling is the only source of the moment term $d/p$. A Gaussian-to-Gaussian comparison then replaces the resulting Gaussian blocks by an i.i.d Gaussian random field , while contributing as error the cost of dependence $(d-1)/2$. Finally, the Vorob'ev--Berkes--Philipp lemma stitches the local couplings into a single extension of the original probability space, on which the local errors are aggregated across dyadic scales.

\subsection{Organization of the paper}
Section~\ref{sec:main-result} states the local approximation conditions,
the strong approximation theorem, worked examples and the optimality result.
Section~\ref{sec:main-proof} proves the main theorem.
Appendices~\ref{sec:helper-results} and~\ref{sec:helper-proofs}
collect auxiliary results and their proofs.
Appendices~\ref{sec:proof-GA} and~\ref{sec:proof-spectral}
give the local Gaussian approximation and Gaussian-to-Gaussian comparison arguments.

\subsection{Notation}\label{sec:notation}
We write $\mathbb N=\{1,2,\ldots\}$ and $[n]=\{1,\ldots,n\}$,
and interpret inequalities between lattice points coordinatewise.
For $m=(m_1,\ldots,m_d)\in\mathbb N^d$, define
\[
  [1,m]=\prod_{j=1}^d\{1,\ldots,m_j\},
  \qquad
  S_X(m)=\sum_{i\in[1,m]}X_i.
\]
Partial sums of other arrays are defined analogously. For a positive
integer $n$, we use $m\le n$ as shorthand for $m\in[n]^d$. For a finite set $Q$, $|Q|$ denotes its cardinality.
For a vector $a\in\mathbb R^D$, $|a|$ denotes its Euclidean norm,
and $|a|_\infty=\max_{1\le j\le D}|a_j|$.
For $h\in\mathbb Z^d$ and a nonempty finite set $Q\subset\mathbb Z^d$,
define
\[
  \dist_\infty(h,Q)=\min_{i\in Q}|h-i|_\infty.
\]
For a rectangle
\[
  Q=\prod_{j=1}^d\{a_j,\ldots,a_j+L_j-1\},
\]
we write $R\preceq Q$ if $R$ is a subrectangle anchored at the
lower-left corner of $Q$, that is,
\[
  R=\prod_{j=1}^d\{a_j,\ldots,b_j\},
  \qquad
  a_j\le b_j\le a_j+L_j-1.
\]

For a random vector $Z$ and $q\ge1$, let
$\|Z\|_q=(\mathbb E|Z|^q)^{1/q}$ and $\|Z\|=\|Z\|_2$.
For a matrix $M$, $\rho^\star(M)$ denotes its largest singular value,
which equals $\sup_{|a|=1}|a^\top Ma|$ when $M$ is symmetric.
For positive semidefinite $\Sigma$, $N(0,\Sigma)$ denotes the centered
Gaussian law with covariance $\Sigma$, including the singular case.
We write $\iu$ for the imaginary unit and $\mathcal L(Z)$ for the
law of $Z$.

For a deterministic sequence $a_n>0$, we write $Z_n=O_{\as}(a_n)$
if, almost surely, $|Z_n|\le C_0a_n$ for all $n\ge n_0$, where
$C_0$ and $n_0$ are finite random variables.
We use the usual deterministic Landau notation.
We write $Z_n=O_{\mathbb P}(a_n)$ if $Z_n/a_n$ is bounded in
probability and $Z_n=o_{\mathbb P}(a_n)$ if $Z_n/a_n$ converges
to zero in probability.
Finally, we write $a_n\lesssim b_n$ if $a_n\le Cb_n$ for some
constant $C>0$, and $a_n\asymp b_n$ if both $a_n\lesssim b_n$
and $b_n\lesssim a_n$ hold.
\section{Strong approximation and optimality}\label{sec:main-result}

\noindent For $d\ge1$, we consider a centered stationary random field
$(X_i)_{i\in\mathbb Z^d}$ with the Bernoulli representation
\begin{equation}\label{eq:bernoulli-shift}
  X_i=g\bigl((\eta_{i-s})_{s\in\mathbb Z^d}\bigr),
  \  i\in\mathbb Z^d,
\end{equation}
where $(\eta_j)_{j\in\mathbb Z^d}$ are independent and identically
distributed real random variables on a probability space
$(\Omega,\mathcal F,\mathbb P)$, and
$g:\mathbb R^{\mathbb Z^d}\to\mathbb R$ is measurable. In general, when $d\geq 2$, due to the lack of spatial ordering, the above representation allows for dependence in all directions.

Following the functional-dependence framework introduced
for time series by \citet{wu2005nonlinear} and developed for random
fields by \citet{el2013central}, we measure dependence through the effect of changing a single
innovation. Let $X_i^\star$ be obtained from $X_i$ by replacing
$\eta_0$ with an independent copy $\eta_0'$, while leaving all other
innovations unchanged. For $q\ge1$, define
\begin{equation}\label{eq:functional-dependence}
  \delta_{i,q}=\|X_i-X_i^\star\|_q,
  \ 
  \Delta_q=\sum_{i\in\mathbb Z^d}\delta_{i,q},
\end{equation}
where $\|\cdot\|_q$ is the usual $L^q$ norm. The coefficient
$\delta_{i,q}$ measures the sensitivity of the observation at $i$ to
the innovation at the origin, and $\Delta_q$ aggregates this influence
over the lattice. 
\subsection{$m$-dependent approximations and assumptions}
We work with the centered field \eqref{eq:bernoulli-shift} and the
functional-dependence coefficients \eqref{eq:functional-dependence}.
The quantitative assumption concerns the dependence remaining after
conditioning on nearby innovations. For $m\in\mathbb N$, define
\begin{equation}\label{eq:local-def}
  \Fcal_i(m)=\sigma(\eta_{i-s}:|s|_\infty\le m),
  \
  X_i^{[m]}=\IE[X_i\mid\Fcal_i(m)].
\end{equation}
The localization radius is $m$: collections of these variables are
independent whenever their index sets have supremum-norm distance
greater than $2m$. Thus $(X_i^{[m]})$ is a finite-range approximation.
Since $X_i^{[m]}$ is obtained by integrating out the innovations outside the window $\{i-s:|s|_\infty\le m\}$, we may and do fix, once and for all, a Borel function $g_m$ such that $X_i^{[m]}=g_m(\eta_{i-s};|s|_\infty\le m)$ almost surely, for every $i\in\Z^d$. All modified versions of $X_i^{[m]}$ appearing below (starred versions, surrogate versions) are obtained by evaluating the same $g_m$ at the modified innovations. Similar to $X_i^\star$, also denote by $(X_i -X_i^{[m]})^\star$ when we replace $\eta_0$ by $\eta_0'$ in $X_i$ and in $X_i^{[m]}$. For $p\ge1$, define the functional-dependence coefficients of this residual field by
\begin{equation}\label{eq:residual-dependence}
  \delta_{i,p}^{[m]}=\norm{(X_i - X_i^{[m]})-(X_i -X_i^{[m]})^\star}_p,
  \ 
  \Delta_p^{[m]}=\sum_{i\in\Z^d}\delta_{i,p}^{[m]}.
\end{equation}

\subsection{Main theorem}
\begin{theorem}\label{thm:main}

Let $d\ge2$, and consider the centered stationary random field
$(X_i)_{i\in\mathbb Z^d}$ of the form \eqref{eq:bernoulli-shift}.
Suppose that $\|X_0\|_p<\infty$ and $\Delta_p<\infty$ for some $p>2$.
Assume that there is a constant
$A>\min\{d(1-2/p),1\}$ such that
\color{black}
\begin{equation}\label{eq:residual-rate-main}
  \Delta_p^{[s]}\lesssim (1+s)^{-A},\ s\ge1.
\end{equation}

Assume further that
$\sigma_\infty^2:=\sum_{h\in\Z^d}\Cov(X_0,X_h)>0$.
Then, on an extension of the original probability space, there exists
a field $(Y_i)_{i\in\Z^d}$ of mutually independent $N(0,1)$ random
variables such that, for every $\eps>0$,
\color{black}
\begin{equation}\label{eq:main-rate}
\max_{m\le n}\abs{S_X(m)-\sigma_\infty S_Y(m)}
=O_{\as}(a_n),
\end{equation}
where
\begin{align}
    a_n:=n^{\max\{ d/p, (d-1)/2 \} } (1 \vee \log n)^{\,3d/2+5/2-1/p}(1\vee \log \log n)^{\,3/2+\eps}. \label{eq:a-n}
\end{align}
\end{theorem}
An arbitrary axis-aligned rectangle in $[n]^d$ can be written by
inclusion--exclusion using at most $2^d$ rectangular prefixes.
Consequently, \eqref{eq:main-rate} controls the error over all such
rectangles, up to this constant factor. In the following, we highlight two aspects of Theorem \ref{thm:main}; the ubiquity of \eqref{eq:residual-rate-main}, and the sharpness of the polynomial terms in $a_n$.
\subsection{\texorpdfstring{Examples\color{black}}{Examples}}\label{ssc:examples}

The following examples verify the condition \eqref{eq:residual-rate-main} for
linear and nonlinear fields with infinite-range dependence. For a real
coefficient array $a=(a_s)_{s\in\mathbb Z^d}$, write
\[
 \|a\|_{\ell^q}=\Big(\sum_s|a_s|^q\Big)^{1/q},\ 
 T_q(m)=\Big(\sum_{|s|_\infty>m}|a_s|^q\Big)^{1/q},
 \  q\in\{1,2\}.
\]

\begin{example}[Two-sided linear fields]
Let the innovations be centered, with $\|\eta_0\|_p<\infty$ for
some $p>2$ and $\Var(\eta_0)>0$. Suppose that
$\|a\|_{\ell^1}<\infty$ and $\sum_s a_s\ne0$, and set
\[
 X_i=\sum_s a_s\eta_{i-s}.
\]
The series converges in $L^p$. Independence and centering give
\[
 X_i^{[m]}=\sum_{|s|_\infty\le m}a_s\eta_{i-s},\ 
 \delta_{i,p}=|a_i|\|\eta_0-\eta'_0\|_p,
 \ 
 \delta_{i,p}^{[m]}=\delta_{i,p}\ind\{|i|_\infty>m\}.
\]
Consequently,
\[
 \Delta_p=\|\eta_0-\eta'_0\|_p\|a\|_{\ell^1},\ 
 \Delta_p^{[m]}=\|\eta_0-\eta'_0\|_p T_1(m),\ 
 \sigma_\infty^2=\Var(\eta_0)\Big(\sum_s a_s\Big)^2>0.
\]
In particular, if
\begin{equation}\label{eq:example-coeff-decay}
 |a_s|\lesssim (1+|s|_\infty)^{-\tau},\ 
 \tau>d+\min\{d(1-2/p),1\},
\end{equation}
then $T_1(m)\lesssim (1+m)^{d-\tau}$, and
Theorem~\ref{thm:main} applies with $A=\tau-d$.
\end{example}

\begin{example}[Quadratic Gaussian linear fields]
Let the innovation $\eta_i$'s be independent standard Gaussians, let $a \neq 0$ and $\|a\|_{\ell^1}= \sum_s |a_s|<\infty$. Define
\[
 L_i=\sum_s a_s\eta_{i-s},\  X_i=L_i^2-\mathbb E L_i^2.
\]
 H\"older's inequality gives $\delta_{i,p}\lesssim |a_i|\|a\|_{\ell^2}$, so
$\Delta_p<\infty$ for every $p>2$.
For $U_{i,m}=\sum_{|s|_\infty\le m}a_s\eta_{i-s}$ and
$V_{i,m}=L_i-U_{i,m}$, independence yields
\[
 X_i^{[m]}=U_{i,m}^2-\mathbb E U_{i,m}^2,\ 
 X_i-X_i^{[m]}=2U_{i,m}V_{i,m}
                   +V_{i,m}^2-\mathbb E V_{i,m}^2.
\]
Replacing $\eta_0$ by $\eta_0'$ gives
\[
(X_i-X_i^{[m]})-(X_i-X_i^{[m]})^\star=\begin{cases}
    2a_i(\eta_0-\eta_0')V_{i,m}, & |i|_\infty\le m,\\
    X_i-X_i^\star, & |i|_\infty>m,
\end{cases}
\]
whence, by H\"older's inequality and Gaussian moment bounds,
\[
\delta_{i,p}^{[m]}\lesssim\begin{cases}
    |a_i|\,T_2(m), & |i|_\infty\le m,\\
    |a_i|\,\|a\|_{\ell^2}, & |i|_\infty>m.
\end{cases}
\]
Summing over $i\in\Z^d$,
\begin{equation}\label{eq:quadratic-residual-bound}
 \Delta_p^{[m]}\lesssim \|a\|_{\ell^1}T_2(m)+\|a\|_{\ell^2}T_1(m).
\end{equation}
Since $T_2(m)\le T_1(m)$, condition
\eqref{eq:example-coeff-decay} again gives
\eqref{eq:residual-rate-main} with $A=\tau-d$.
\end{example}
\color{black}

\subsection{Optimality of the polynomial rate}\label{ssc:optimality}
For a coupling of $(X_i)$ with mutually independent standard Gaussian
variables $(Y_i)$, write
\begin{equation}\label{eq:approximation-error}
 D_n=\max_{m\in[n]^d}
       |S_X(m)-\sigma_\infty S_Y(m)|.
\end{equation}
The following example show that the polynomial exponent in
Theorem~\ref{thm:main} is sharp. There are two aspects to this sharpness: the exponent $d/p$ is the natural one, and has already appeared in corresponding strong approximation results for independent random field \citep{rio1993a, rio1997}, and its $d=1$ version has been proved for stationary observations \citep{berkes2014kmt}. The rate $(d-1)/2$ has also appeared in the literature \citep{rio1993a}, but in the context of strong approximation over classes larger than the class of rectangles. Even though one might think this term to be some sort of cost of dimensionality, \cite{rio1997} succeeded in removing this rate for $d=2$ for independent random fields. Perhaps surprisingly, this term reappears when we go beyond independence to stationary random fields. In the following, we show that the term $(d-1)/2$ is the \textit{cost of dependence}, and is unavoidable.

\begin{example}
    There exists a centered finite-range Gaussian field satisfying
the assumptions of Theorem~\ref{thm:main} for every $p>2$ such that
no coupling with $(Y_i)$ has
$D_n=o_{\mathbb P}(n^{(d-1)/2})$.
To see this, let
$X_i=\eta_i+\theta\eta_{i-e_1}$, where the innovations are independent
standard Gaussians, $e_1=(1,0,\ldots,0)$ and $0<\theta<1$.
This finite-range field satisfies the assumptions for every $p>2$,
has long-run variance $(1+\theta)^2$, and has spectral density
bounded away from zero.
For the thin rectangle with endpoint $m_n=(1,n,\ldots,n)$,
\[
 \frac{S_X(m_n)}{n^{(d-1)/2}}\sim N(0,1+\theta^2),
 \ 
 \frac{\sigma_\infty S_Y(m_n)}{n^{(d-1)/2}}
       \sim N(0,(1+\theta)^2).
\]
Since $m_n\in[n]^d$,
\[
\frac{|S_X(m_n)-\sigma_\infty S_Y(m_n)|}{n^{(d-1)/2}}\le\frac{D_n}{n^{(d-1)/2}}.
\]
If $D_n=o_{\IP}(n^{(d-1)/2})$, the left-hand side would tend to zero in probability, and Slutsky's theorem would force the variances of the two Gaussian laws in the previous display to coincide, i.e., $1+\theta^2=(1+\theta)^2$, or equivalently $\theta=0$. This contradicts $\theta>0$ and establishes the sharpness of the polynomial exponent
$(d-1)/2$ in \eqref{eq:a-n} as a cost of dependence.
\end{example}

\section{Proof of the main theorem}\label{sec:main-proof}
The proof of Theorem \ref{thm:main} proceeds in five steps, which we briefly outline for readers' convenience.
\begin{itemize}
    \item \textbf{Blocking} (\S\ref{ssc:blocking}): we fix a permanent dyadic partition $\cQ$ of $\N^d$ into cubes $Q$, and replace $(X_i)_{i \in Q}$ by constructing i.i.d. copies $\BX_Q:= (X_i^Q)_{i \in Q}$ so that $\BX_Q$ are independent across $Q$. The cost of this replacement is controlled by Proposition \ref{prop:Q-blocks}.
    \item \textbf{Local Gaussian approximation} (\S\ref{ssc:local-GA}): on each cube $Q$, Theorem \ref{thm:GA} couples the block $\BX_Q$ with a Gaussian vector $\bG_Q:=(G_i^Q)_{i \in Q}$ with $\Cov(\BX_Q)= \Cov(\bG_Q)$, uniformly over all lower-left sub-rectangles.
    \item \textbf{From stationary Gaussian field to an i.i.d one} (\S\ref{ssc:spectral}): Theorem \ref{thm:spectral} couples a stationary Gaussian field $(G^{\mathrm{sp}}_i)_{i\in\Z^d}$ (so that $(G^{\mathrm{sp}}_i)_{i \in Q} \overset{d}{=} \bG_Q$) to an i.i.d. random field $\sigma_{\infty} Y_i^{\mathrm{sp}}$, where $Y_i^{\mathrm{sp}}\sim N(0,1)$.
    \item \textbf{Aggregation} (\S\ref{ssc:proof-main}): the three local errors are summed over the cubes meeting $[1,n]^d$ to obtain the final error in our invariance principle.
    \item \textbf{Enlarged probability space} (\S\ref{ssc:gluing}): The aggregation step above requires a common probability space. We invoke the Vorob'ev--Berkes--Philipp theorem consigns the above blockwise couplings to a single extension of $(\Omega,\Fcal,\IP)$.
\end{itemize}
In the following, we discuss each of these steps seriatim.
\subsection{Blocking}\label{ssc:blocking}
Theorem~\ref{thm:main} asks for a single coupling that controls rectangular partial sums at every observation scale. In contrast, Gaussian couplings are usually constructed for finitely many observations at a time. Such local couplings can be assembled into one global coupling only if the pieces on which they act are independent. On the other hand, the original field restricted to disjoint blocks are not independent, since each $X_i$ depends on innovations outside any block containing $i$. In this section, we replace each block by an independent surrogate with the same law, and control the error suitably. Following the dyadic blocking approach of \cite{berkes2014kmt}, we first fix a partition of the lattice that does not depend on the observation scale $n$; this permanence is what allows a single coupling to serve all observation scales simultaneously. For $k\ge0$, let $$I_{k,0}=[2^k]= \{1, \ldots, 2^k\}, \ \text{and } I_{k,1}=\{2^k+1, \ldots, 2^{k+1}\}.$$ Define, for every binary vector $e=(e_1, \ldots, e_d)\in \{0,1\}^d \setminus \{0\}$,
\[ Q_{k,e}:=\prod_{j=1}^d I_{k,e_j}. \]
Moreover, let $Q_*= \{(1,\ldots, 1)\}$. Then
\begin{align} \label{eq:cq}
    \cQ
  :=\{Q_*\}\cup
  \{Q_{k,e}:k\ge0,\ e\in\{0,1\}^d\setminus\{0\}\}
\end{align}
  constitutes a countable, disjoint partition of $\N^d$ consisting entirely of cubes. Indeed, if $i\in\N^d$ satisfies $|i|_\infty\ge 2$, then there is a unique $k\ge0$ with $2^k<|i|_\infty\le 2^{k+1}$, and $i\in Q_{k,e}$ for the unique $e$ given by $e_j=\ind\{i_j>2^k\}$; since every $Q_{k,e}$ is contained in $[2^{k+1}]^d\setminus[2^k]^d$, cells belonging to distinct values of $k$ are disjoint. In particular, for each $k\ge 0$ and $e\in \{0,1\}^d \setminus \{0\}$, $Q_{k,e}$ has sides of length $2^k$, and
  \begin{equation}\label{eq:shells}
      [2^K]^d=Q_*\cup\bigcup_{k=0}^{K-1}\ \bigcup_{e\ne 0}Q_{k,e},\  K\ge 1.
  \end{equation}
  Let $\cQ_0:=\{Q_*\}\cup\{Q_{0,e}:e\ne0\}$ and $\cQ_k:=\{Q_{k,e}:e\ne0\}$, $k\ge1$; thus $|\cQ_k|\le 2^d$, and every cell in $\cQ_k$ has side length $2^k$. Consider the product extension
  \[(\Omega_0,\Fcal_0,\IP_0)
  =\Big(
      \Omega\times\Omega^{\cQ},
      \Fcal\otimes\Fcal^{\otimes\cQ},
      \IP\otimes\IP^{\otimes\cQ}
    \Big), \]
  on which we can define, for each $Q \in \cQ$, an i.i.d.\ copy $\{\eta^{(Q)}_j\}_{j\in \Z^d}$ of $\{\eta_j\}_{j\in \Z^d}$, these copies being mutually independent and independent of the original innovations. For $Q \in \cQ$, $h\in \Z^d$, define
  \begin{align} \label{eq:Q-blocks}
      \eta_h^Q = \begin{cases}
      \eta_h,& h\in Q,\\
      \eta^{(Q)}_h, & h \notin Q,
  \end{cases} \ \text{and}, \ X_i^Q = g(\eta_{i-s}^Q; \ s\in \Z^d), \ i\in \Z^d.
  \end{align}
Let $\BX^Q= \{X_i^Q\}_{i \in \Z^d}$, and denote the corresponding surrogate block by $\BX_Q:=(X^Q_i)_{i\in Q}\in\R^Q$. By construction of $\{\eta^Q_j\}_{j \in \Z^d}$ and using that $\cQ$ constitutes a disjoint partition, it immediately follows that 
\begin{equation}\label{eq:block-independence}
    (\BX^Q)_{Q \in \cQ} \text{ are mutually independent under }\IP_0.
\end{equation}
Moreover, since $(\eta^Q_h)_{h\in\Z^d}$ is again i.i.d.\ with the law of $\eta_0$, it follows that $\Lcal(\BX^Q)=\Lcal((X_i)_{i\in\Z^d})$, and in particular $\Lcal(\BX_Q)=\Lcal((X_i)_{i\in Q})$ for every $Q\in\cQ$. Thus, each surrogate block reproduces the law of the original block exactly, while \eqref{eq:block-independence} decouples distinct blocks. The price is the discrepancy between $X_i$ and $X_i^Q$, which share the innovations indexed by $Q$ but not those outside it. By \eqref{eq:residual-rate-main}, $X_i$ depends only weakly on innovations far from $i$, so this discrepancy is concentrated near the boundary of $Q$. Proposition~\ref{prop:Q-blocks}, proved in Appendix~\ref{sec:helper-proofs}, makes this precise: for $Q\in\cQ_k$,
\begin{align}
    \big\| \max_{R \preceq Q} |\sum_{i\in R}(X_i -  X_i^Q)|\big\|_p \lesssim u_A(k),\ u_A(k):=\begin{cases}
        2^{k(d-1)/2},& A>1/2,\\
        2^{k(d-1)/2} (k+1), & A=1/2,\\
        2^{k(d/2-A)}, & 0 < A < 1/2.
        \end{cases}  \label{eq:Prop-1}
\end{align}
Denote $U_Q:=\max_{R \preceq Q} |\sum_{i\in R}(X_i -  X_i^Q)|$. Then, \eqref{eq:Prop-1} along with $|\cQ_k|\le 2^d$ and Markov's inequality instructs that
\[ \sum_{k=0}^{\infty} \IP_0\Big(
    \max_{Q\in\cQ_k}U_Q
    >u_A(k)\sqrt{k+1}
  \Big)
  \lesssim \sum_{k=0}^{\infty}(k+1)^{-p/2}<\infty,  \]
since $p>2$, which, via the Borel--Cantelli lemma, implies that
\begin{equation}\label{eq:boundary-as}
  \max_{Q\in\cQ_k}U_Q
  =O_{\as}\!\Big(u_A(k)\sqrt{k+1}\Big), \   k\to\infty.
\end{equation}
The bound \eqref{eq:boundary-as} comes from the surface of a cube of side $2^k$. The outermost layer of such a cube contains $O(2^{k(d-1)})$ sites, all of which echo the effect of replacing the innovations across the boundary. This is the first of the two places in the proof where the surface term $n^{(d-1)/2}$ in \eqref{eq:main-rate} arises. With independent surrogate blocks with the marginal laws in hand, we now couple each of them with a Gaussian vector.
\subsection{A local Gaussian approximation}\label{ssc:local-GA}
With the blocks decoupled, each surrogate block $\BX_Q$ can be treated in isolation. In this section, we couple it with a centered Gaussian vector $\bG_Q$ of the same covariance, $\Cov(\bG_Q)=\Sigma_Q$, such that $\max_{R\preceq Q}|\sum_{i\in R}(X^Q_i-G_{Q,i})|$ is small. By stationarity, it suffices to treat the cube $[2^k]^d$.
\begin{theorem} \label{thm:GA}
Grant the assumptions of Theorem \ref{thm:main}. Let $Q:=[2^k]^d$ for some $k\ge 1$. Denote $r= p \wedge 4$, and $\nu_A:= 1 + \sum_{\ell=1}^k2^{\ell(d(1-2/r) - A)}.$ Then, there exists a constant $C>0$ independent of $k$ and $x$, such that for every $x\in (0,1/2)$ the following holds: on a potentially enriched probability space, there exist a random vector $\tilde{\BX}^Q=(\Tilde{X}^Q_i)_{i\in Q} \overset{d}{=} ({X}_i)_{i\in Q}$ and a centered Gaussian vector $\bG^Q=(G^Q_i)_{i\in Q}$ with $\Cov(\bG^Q)=\Cov(({X}_i)_{i\in Q})$, such that
\[\IP\Big(\max_{R \preceq Q} \big|\sum_{i \in R}(\Tilde{X}^Q_i - G^Q_i)\big| > C x^{-3/2} 2^{kd/r} \nu_A^{1/2} (k+1)^{3d/2 + 1- 1/r}\Big)\le x. \]
\end{theorem}
Note that by stationarity, Theorem \ref{thm:GA} holds for every $Q\in\cQ_k$. Let $\{x_k\}_{k\ge0}\subset(0,1/2)$ be a sequence with $\sum_kx_k<\infty$; its choice will be specified in \S\ref{ssc:proof-main}. For each $Q\in\cQ$, put
\[
\mu_Q:=\Lcal((X_i)_{i\in Q}),\ 
\Sigma_Q:=\Cov((X_i)_{i\in Q}),\ 
\nu_Q:=N(0,\Sigma_Q).
\]
For $Q\in\cQ_k$ with $k\ge1$, choose the finite-dimensional joint law
\[
\pi_Q:=\Lcal(\widetilde{\BX}_Q,\bG_Q)
\ \text{on }\R^Q\times \R^Q
\]
provided by Theorem~\ref{thm:GA} with $x=x_k$. Its marginals are $\mu_Q$ and $\nu_Q$, respectively. For the blocks $Q\in\cQ_0$, choose $\pi_Q:=\mu_Q\otimes\nu_Q$. On the common probability space of \S\ref{ssc:proof-main}, let $\bG_Q$ be such that $(\BX_Q,\bG_Q)$ has law $\pi_Q$, and write
\[
V_Q:=\max_{R\preceq Q}\Big|\sum_{i\in R}(X_i^Q-G_{Q,i})\Big|.
\]
Set
\begin{align}
    b_k:=C x_k^{-3/2}2^{kd/r}\nu_A(k)^{1/2}(k+1)^{3d/2+1-1/r}, \label{eq:b_k}
\end{align}
with $C$ as in Theorem~\ref{thm:GA}. Theorem~\ref{thm:GA} along with $|\cQ_k|\le 2^d$ instructs that
\[ \sum_{k=1}^{\infty}\tilde\IP\Big(\max_{Q\in\cQ_k}V_Q>b_k\Big)\le 2^d\sum_{k=1}^\infty x_k<\infty,\]
which, via the Borel--Cantelli lemma, implies that
\begin{equation}\label{eq:VQ-as}
  \max_{Q\in\cQ_k}V_Q
  =O_{\as}\!\Big(b_k\Big), \   k\to\infty.
\end{equation}
We remark that the moment exponent $d/p$ enters the proof through \eqref{eq:VQ-as}; for $p\le 4$, $b_k$ grows as $2^{kd/p}$, up to the dependence factor $\nu_A(k)^{1/2}$ and logarithmic factors. For $p>4$, the method saturates at $2^{kd/4}$, which is absorbed by the surface term $2^{k(d-1)/2}$, since $d\ge2$. 

Note that the Gaussian vectors $\bG_Q$, however, carry the covariance $\Sigma_Q$ of the original field rather than that of white noise, and \S\ref{ssc:spectral} closes this gap.

\subsection{From a stationary Gaussian field to an i.i.d one}\label{ssc:spectral}
The Gaussian vectors $\bG_Q$ of \S\ref{ssc:local-GA} inherit the covariance $\Sigma_Q$ of the original field, whereas the target in Theorem~\ref{thm:main} is white noise scaled by $\sigma_\infty$. This section bridges the two: within each block, we couple a stationary Gaussian field with autocovariance $\gamma(h)=\Cov(X_0,X_h)$ to $\sigma_\infty$ times an i.i.d.\ $N(0,1)$ field, uniformly over lower-left sub-rectangles. Denote
 Denote
\begin{equation}\label{eq:wA}
    w_A(k):=\begin{cases}
        2^{k(d-2A)/2}(k+1)^{1/2}, & 0<A<1/2,\\
        2^{k(d-1)/2}(k+1), & A=1/2,\\
        2^{k(d-1)/2}(k+1)^{1/2}, & A>1/2.
    \end{cases}
\end{equation}
The following theorem compares a dependent Gaussian field with long run variance $\sigma_{\infty}^2$, to an i.i.d. Gaussian random field with variance $\sigma_{\infty}^2$.
\begin{theorem}\label{thm:spectral}
Grant the assumptions of Theorem \ref{thm:main}. On some probability space $(\Omega^{\mathrm{sp}},\Fcal^{\mathrm{sp}},\IP^{\mathrm{sp}})$ there exist centered, jointly Gaussian, stationary fields $(G^{\mathrm{sp}}_i)_{i\in\Z^d}$ and $(Y^{\mathrm{sp}}_i)_{i\in\Z^d}$ such that $\Cov(G^{\mathrm{sp}}_i,G^{\mathrm{sp}}_j)=\gamma(i-j)$, $(Y^{\mathrm{sp}}_i)_{i\in\Z^d}$ are i.i.d.\ $N(0,1)$, and, for a constant $C>0$, for every $k\ge0$ and every translate $Q$ of $[2^k]^d$,
\[\IP^{\mathrm{sp}}\Big(\max_{R\preceq Q}\big|\sum_{i\in R}(G^{\mathrm{sp}}_i-\sigma_\infty Y^{\mathrm{sp}}_i)\big|>C\,w_A(k)\Big)\le 2\cdot 4^{-k}.\]
\end{theorem}
The proof is given in Appendix~\ref{sec:proof-spectral}. For each $Q\in\cQ$, define the probability measure
\[
\rho_Q:=\Lcal_{\IP^{\mathrm{sp}}}\big((G_i^{\mathrm{sp}})_{i\in Q},(Y_i^{\mathrm{sp}})_{i\in Q}\big)
\ \text{on }\R^Q\times \R^Q.
\]
Its marginals are $\nu_Q$ and $N(0,I_Q)$; in particular, its first marginal coincides with the second marginal of $\pi_Q$, which is precisely what allows the two couplings to be linked in \S\ref{ssc:gluing}. As in \S\ref{ssc:local-GA}, only the law $\rho_Q$ enters the subsequent argument, and the stationary fields of Theorem~\ref{thm:spectral} serve merely to construct it. On the common probability space of \S\ref{ssc:proof-main}, let $\BY_Q=(Y_{Q,i})_{i\in Q}$ be such that $(\bG_Q,\BY_Q)$ has law $\rho_Q$, and write $W_Q:=\max_{R\preceq Q}|\sum_{i\in R}(G_{Q,i}-\sigma_\infty Y_{Q,i})|$. Theorem~\ref{thm:spectral} along with $|\cQ_k|\le 2^d$ instructs that
\[ \sum_{k=0}^{\infty}\tilde\IP\Big(\max_{Q\in\cQ_k}W_Q>C\,w_A(k)\Big)\le 2^{d+1}\sum_{k=0}^\infty 4^{-k}<\infty,\]
which, via the Borel--Cantelli lemma, implies that
\begin{equation}\label{eq:WQ-as}
  \max_{Q\in\cQ_k}W_Q
  =O_{\as}\!\Big(w_A(k)\Big), \   k\to\infty.
\end{equation}
This completes the list of local error bounds, alongside \eqref{eq:boundary-as} and \eqref{eq:VQ-as}. For every $A>0$, $w_A(k)$ has the same polynomial order as the block-replacement error $u_A(k)$ in \eqref{eq:Prop-1}, namely $2^{k(d-1)/2}$ when $A>1/2$. It remains to sum the three local errors over the blocks meeting $[1,n]^d$, which we do next.
\subsection{Aggregation}\label{ssc:proof-main}
Each of the preceding three sections produced a local error bound on the cells of a single dyadic shell: the block-replacement error \eqref{eq:boundary-as}, the local Gaussian approximation error \eqref{eq:VQ-as} and the dependent-to-independent Gaussian comparison error \eqref{eq:WQ-as}, of orders $u_A(k)\sqrt{k+1}$, $b_k$ and $w_A(k)$, respectively. In this section, we sum these errors over the cells meeting $[1,n]^d$. As announced in the outline of Section~\ref{sec:main-proof}, we take for granted a common probability space $(\tilde\Omega,\tilde\Fcal,\IP)$ on which the original field, the surrogate blocks $\BX_Q$, and the vectors $\bG_Q$ and $\BY_Q$ are defined simultaneously, with $(\BX_Q,\bG_Q)\sim\pi_Q$ and $(\bG_Q,\BY_Q)\sim\rho_Q$ for every $Q\in\cQ$. Two features of the blocking make the aggregation straightforward. First, by \eqref{eq:shells}, the cube $[1,n]^d$ is contained in $[2^K]^d$, $K:=\lceil\log_2 n\rceil$, which is the disjoint union of the cells in $\cQ_0,\ldots,\cQ_{K-1}$; each $\cQ_k$ contains at most $2^d$ cells, and every prefix rectangle $[1,m]$ intersects each cell in a lower-left sub-rectangle. Hence, by \eqref{eq:boundary-as}, \eqref{eq:VQ-as} and \eqref{eq:WQ-as}, almost surely,
\[
\max_{m\le n}|S_X(m)-\sigma_\infty S_Y(m)|\lesssim\sum_{k=0}^{K-1}\big(u_A(k)\sqrt{k+1}+b_k+w_A(k)\big).
\]
Second, we  choose the summable sequence $\{x_k\}$ in \eqref{eq:b_k} suitably to ensure all  three errors: $u_A(k)\sqrt{k+1}$, $b_k$ and $w_A(k)$ match the powers of $2^k$, the side length of the cells in $\cQ_k$. The sum over $k\le K-1$ is therefore $\lesssim$ its last term, which corresponds to cells of side length $2^{K-1}$, comparable to $n$. Consequently, the polynomial exponent in \eqref{eq:main-rate} arises as the largest of the exponents of the three local errors. The remainder of this section fixes $\{x_k\}$ and codifies the above strategy.

\subsubsection{Step 1}  Let $n\ge 2$, and let $K:=\lceil\log_2 n\rceil$, so that $2^{K-1}<n\le 2^K$. For any cell $Q \in \cQ_k$ for some $k\in \{0,1,\ldots, K-1\}$, let $R:=[1,m]\cap Q\preceq Q$. Observe that
\[\sum_{i\in R}(X_i-\sigma_\infty Y_i)=\sum_{i\in R}(X_i-X^Q_i)+\sum_{i\in R}(X^Q_i-G_{Q,i})+\sum_{i\in R}(G_{Q,i}-\sigma_\infty Y_i).\]
Recall $U_Q$, $V_Q$ and $W_Q$ from \S\ref{ssc:blocking}, \S\ref{ssc:local-GA} and \S\ref{ssc:spectral}. Immediately, \eqref{eq:boundary-as}, \eqref{eq:VQ-as} and \eqref{eq:WQ-as} let us write
\begin{equation}\label{eq:local-as}
   \max_{m\le n}|S_X(m)-\sigma_\infty S_Y(m)| \le  \sum_{k=0}^{K-1}\sum_{Q\in\cQ_k}(U_Q+V_Q+W_Q) = O_{\text{a.s.}}(\sum_{k=0}^{K-1}\varpi_k),\  k\ge 0,
\end{equation}
where 
\[\varpi_k:=u_A(k)\sqrt{k+1}+b_k+w_A(k), \]
with $u_A(k)$ as in \eqref{eq:Prop-1}, $b_k$ from \eqref{eq:b_k} and $w_A(k)$ as in \eqref{eq:wA}. Note that each of the three sums in $\varpi_k$ is bounded by a term of the form $2^{k\alpha}\varphi(k)$, with $\alpha>0$ and $\varphi(k)=(k+1)^{a}\{\log(k+3)\}^{b}$, $a,b\ge0$. Therefore, for each of these terms,
\begin{align}\label{eq:geom-sum}
    \sum_{k=0}^{K-1}2^{k\alpha}\varphi(k)\le \varphi(K)\sum_{k=0}^{K-1}2^{k\alpha}\lesssim 2^{K\alpha}\varphi(K)\lesssim n^{\alpha}(\log n)^a(\log\log n)^b,\   n\ge 16.
\end{align}
\subsubsection{Step 2.}
It remains to identify the exponents $\alpha$, $a$ and $b$ for each of the three local errors. Let $c:=d(1-2/r)$, and $\beta=\max\{d/p,(d-1)/2\}$. Observe that, by the assumption of Theorem~\ref{thm:main},
\begin{equation}\label{eq:beta-identity}
    d-2\beta=\min\{d(1-2/p),1\}<A.
\end{equation}
If $\alpha<\beta$, then $n^{\alpha}(\log n)^a(\log\log n)^b=O(n^\beta)$ for all $a,b\ge0$; hence $a$ and $b$ matter only when $\alpha=\beta$. In every case below, $\alpha>0$. Fix
\begin{equation}
    x_k:=\frac14\big[(k+1)\log(k+3)\,(\log\log(k+16))^2\big]^{-1},\  k\ge0. \nonumber
\end{equation}
Then $x_k\in(0,1/4]$, $\sum_k x_k<\infty$, and, for every $\eps>0$,
\begin{equation}\label{eq:xk-prop}
    x_k^{-3/2}\lesssim (k+1)^{3/2}(\log(k+3))^{3/2+\eps},\  k\ge0.
\end{equation}
\begin{enumerate}[label=(\roman*)]
    \item \emph{Block-replacement error.} By \eqref{eq:Prop-1},
    \[
    u_A(k)\sqrt{k+1}=
    \begin{cases}
        2^{k(d-1)/2}(k+1)^{1/2}, & A>1/2,\\
        2^{k(d-1)/2}(k+1)^{3/2}, & A=1/2,\\
        2^{k(d/2-A)}(k+1)^{1/2}, & A<1/2.
    \end{cases}
    \]
    Since $(d-1)/2\ge d/2-A$ if and only if $A\ge1/2$, \eqref{eq:geom-sum} yields
   \[
\sum_{k=0}^{K-1}u_A(k)\sqrt{k+1} \lesssim \kappa_n^{\mathrm{rep}},\
\kappa_n^{\mathrm{rep}}:=n^{\max\{(d-1)/2,\,d/2-A\}}(\log n)^{3/2}.
\]
    \item \emph{Local Gaussian error.} By \eqref{eq:b_k} and \eqref{eq:xk-prop},
    \[
    b_k \lesssim \, 2^{kd/r}\nu_A(k)^{1/2}(k+1)^{3d/2+5/2-1/r}(\log(k+3))^{3/2+\eps},
    \]
    where 
    \[    \nu_A(k)\lesssim \begin{cases}
        1, & A>c,\\
        k+1, & A=c,\\
        2^{k(c-A)}, & A<c.
    \end{cases}\]
   Consequently, \eqref{eq:geom-sum} yields $\sum_{k=0}^{K-1}b_k\lesssim \kappa_n^{\mathrm{GA}}$ with
    \[
\kappa_n^{\mathrm{GA}}:=
\begin{cases}
    n^{d/r}(\log n)^{3d/2+5/2-1/r}(\log\log n)^{3/2+\eps}, & A>c,\\
    n^{d/r}(\log n)^{3d/2+3-1/r}(\log\log n)^{3/2+\eps}, & A=c,\\
    n^{(d-A)/2}(\log n)^{3d/2+5/2-1/r}(\log\log n)^{3/2+\eps}, & A<c.
\end{cases}
\]
    \item \emph{Dependent-to-independent Gaussian comparison error.} By \eqref{eq:Prop-1} and \eqref{eq:wA}, $w_A(k)\le u_A(k)\sqrt{k+1}$ for every $k\ge0$. Consequently, by (i), $\sum_{k=0}^{K-1}w_A(k)\lesssim \kappa_n^{\mathrm{rep}}$.
\end{enumerate}
Collecting (i)--(iii), almost surely, for all $n\ge16$,
\[
\max_{m\le n}|S_X(m)-\sigma_\infty S_Y(m)|\lesssim \max\{\kappa_n^{\mathrm{rep}},\kappa_n^{\mathrm{GA}}\}.
\]
Now we control $\max\{\kappa_n^{\mathrm{rep}},\kappa_n^{\mathrm{GA}}\}$ by $a_n$ from \eqref{eq:a-n}. We first consider the polynomial exponents $(d-1)/2, d/2-A, d/r$ and $(d-A)/2$. By \eqref{eq:beta-identity}, $A>d-2\beta\ge(d-2\beta)/2$, and therefore, $$\max\{d/2-A,  (d-A)/2\}<\beta.$$ For the other two exponents,
\[
\max\{(d-1)/2,\,d/r\}=\max\{(d-1)/2,\,d/p\}=\beta,
\]
since $r=p$ for $p\le4$, while $d/p<d/4\le(d-1)/2$ for $p>4$. Hence we recover the polynomial exponent from $a_n$ through $\max\{\kappa_n^{\mathrm{rep}},\kappa_n^{\mathrm{GA}}\}$. Moreover, when $A=c$, since $\min\{d(1-2/p),1\}\ge\min\{c,1\}$, \eqref{eq:beta-identity} forces $c>1$, i.e.\ $d/r<(d-1)/2\le\beta$. Therefore, in this case, $\kappa_n^{GA}=O(n^\beta)$.

We next consider the logarithmic factors, focusing on the case $A\ne c$. The maximum logarithmic factor in $\kappa_n^{\mathrm{GA}}$ is $(\log n)^{3d/2+5/2-1/r}(\log\log n)^{3/2+\eps}$, which dominates the factor $(\log n)^{3/2}$ of $\kappa_n^{\mathrm{rep}}$. Consequently,
\[
\max\{\kappa_n^{\mathrm{rep}},\kappa_n^{\mathrm{GA}}\}\lesssim n^{\beta}(\log n)^{3d/2+5/2-1/r}(\log\log n)^{3/2+\eps}\le a_n,
\]
where the last step uses $1/r\ge1/p$, with equality if and only if $p\le4$. Therefore, almost surely, for all $n\ge16$,
\[
\max_{m\le n}|S_X(m)-\sigma_\infty S_Y(m)|\lesssim a_n.
\]
Collecting (i)-(iii), almost surely, for all $n\ge16$,
\[\max_{m\le n}|S_X(m)-\sigma_\infty S_Y(m)|\lesssim n^{\beta}(\log n)^{3d/2+5/2-1/p}(\log\log n)^{3/2+\eps}.\]

Since $\eps>0$ is arbitrary and $(Y_i)$ does not depend on $\eps$, this proves \eqref{eq:main-rate}, provided that \eqref{eq:boundary-as}, \eqref{eq:VQ-as} and \eqref{eq:WQ-as} hold simultaneously on one probability space. Thus far, we have taken such a space for granted. The blocking error $U_Q$ compares the original field with its surrogate blocks on $(\Omega_0,\Fcal_0,\IP_0)$, whereas $V_Q$ and $W_Q$ are, at this stage, defined only through the coupling laws $\pi_Q$ and $\rho_Q$. In \S\ref{ssc:gluing}, we link $\pi_Q$ and $\rho_Q$ along their common Gaussian marginal $\nu_Q$, take the product of the resulting three measures over $Q\in\cQ$, and transfer that product measure to the prescribed surrogate blocks $\BX_Q$. The resulting extension $(\tilde\Omega,\tilde\Fcal,\tilde\IP)$ carries the original field, the surrogate blocks, and the vectors $\bG_Q$ and $\BY_Q$ simultaneously, with the pair laws prescribed in \eqref{eq:final-laws}. Since each Borel--Cantelli argument above involves only these pair laws, together with $\IP_0$ for \eqref{eq:boundary-as}, all three almost sure bounds hold on this single space. 

\subsection{Constructing a common probability space}\label{ssc:gluing}
This section turns the abstract coupling laws of \S\ref{ssc:local-GA}--\ref{ssc:spectral} into random variables defined on a single extension of the original probability space. Recall that $(\Omega_0,\Fcal_0,\IP_0)$ from \S\ref{ssc:blocking} carries (a) the original $\eta_j$'s and its i.i.d. copies $\{\eta^{(Q)}_j\}$; and (b) the independent surrogate blocks $\BX_Q$, $Q\in\cQ$.

For $Q\in\cQ$,  recall 
\begin{equation}
    \mu_Q:=\Lcal((X_i)_{i\in Q})=\Lcal_{\IP_0}(\BX_Q),\  \nu_Q:=N(0,\Sigma_Q). \nonumber
\end{equation}
We combine the twin families of probability measures on $\R^Q\times \R^Q$ introduced in \S\ref{ssc:local-GA} - \S\ref{ssc:spectral}: 
\begin{itemize}
    \item the couplings $\pi_Q$, with marginals $\mu_Q$ and $\nu_Q$;
    \item the couplings $\rho_Q$, with marginals $\nu_Q$ and $N(0,I_Q)$.
\end{itemize}
This \textit{linking} is provided by the Vorob'ev--Berkes--Philipp coupling \cite{dudley1999uniform}, restated as Lemma \ref{lem:gluing}; in particular, for each $Q\in\cQ$, we obtain a probability measure $\tau_Q$ on $(\R^Q)^{\otimes 3}$ such that, with $(x_Q,g_Q,y_Q)$ denoting the coordinate map,
\begin{equation}
    \Lcal_{\tau_Q}(x_Q,g_Q)=\pi_Q,\  \Lcal_{\tau_Q}(g_Q,y_Q)=\rho_Q. \nonumber
\end{equation}
At this stage, $\tau_Q$ is still an abstract measure, and its first coordinate $x_Q$ is not yet tied to the block $\BX_Q$, which lives on $(\Omega_0,\Fcal_0,\IP_0)$. To tie them, we first consider the product of measures $\tau_Q$ over $Q\in\cQ$, and then lift the product onto the space $(\Omega_0,\Fcal_0,\IP_0)$, so that the $x$-coordinates are realized by the blocks $\BX_Q$ themselves, while the $(g,y)$-coordinates are attached on an extension. 

To that end, we proceed in three steps.
\begin{enumerate}[label=(\roman*)]
    \item \emph{Product-measure marginals.} Since $\cQ$ is countable, the product measure $\tau:=\bigotimes_{Q\in\cQ}\tau_Q$ exists on $\mathbb{G}\times\mathbb{H}$, where
    \[\mathbb{G}:=\prod_{Q\in\cQ}\R^Q,\  \mathbb{H}:=\prod_{Q\in\cQ}(\R^Q\times \R^Q)\]
    store the $x$-coordinates and the $(g,y)$-coordinates, respectively. The $\mathbb{G}$-marginal of $\tau$ is $\bigotimes_Q\mu_Q$. Let $\BX:=(\BX_Q)_{Q\in\cQ}$, where $\BX_Q=(X_i^Q)_{i\in Q}$ are the surrogate blocks from \eqref{eq:Q-blocks}. By \eqref{eq:block-independence}, $\Lcal_{\IP_0}(\BX)=\bigotimes_Q\mu_Q$. 
    \item \emph{Lifting of measure.} By (ii), Lemma \ref{lem:transfer} applies to $\BX$, $\mathbb{G}$, $\mathbb{H}$ and $\tau$. It yields a probability measure $\tilde\IP$ on $(\tilde\Omega,\tilde\Fcal):=(\Omega_0\times\mathbb{H},\Fcal_0\otimes\mathcal{B}(\mathbb{H}))$ (whose projection onto $(\Omega_0,\Fcal_0)$ is $\IP_0$) such that $\Lcal_{\tilde\IP}(\BX,Z)=\tau$, where $Z=((\bG_Q,\BY_Q))_{Q\in\cQ}$ denotes the $\mathbb{H}$-coordinate. Here $\bG_Q=(G_{Q,i})_{i\in Q}$ and $\BY_Q=(Y_{Q,i})_{i\in Q}$ are the Gaussian vectors entering $V_Q$ and $W_Q$ in \S\ref{ssc:local-GA} and \S\ref{ssc:spectral}, now realized on the same space as the original field. By construction,
    \begin{equation}\label{eq:final-laws}
        \Lcal_{\tilde\IP}(\BX_Q,\bG_Q)=\pi_Q,\  \Lcal_{\tilde\IP}(\bG_Q,\BY_Q)=\rho_Q,
    \end{equation}
    and the triples $(\BX_Q,\bG_Q,\BY_Q)$, $Q\in\cQ$, are mutually independent.
    \item \emph{White noise on $\Z^d$.} Define
    \[Y_i:=Y_{Q,i}\  \text{for the unique } Q\in\cQ \text{ containing } i,\  i\in\N^d.\]
    Since $\BY_Q\sim N(0,I_Q)$ are independent across $Q$, $(Y_i)_{i\in\N^d}$ are i.i.d.\ $N(0,1)$. By one further product extension, we append an independent i.i.d.\ $N(0,1)$ family $(Y_i)_{i\in\Z^d\setminus\N^d}$, and maintain the notation $(\tilde\Omega,\tilde\Fcal,\tilde\IP)$.
\end{enumerate}
In particular, $(\tilde\Omega,\tilde\Fcal,\tilde\IP)$ is the common probability space taken for granted in \S\ref{ssc:blocking}--\S\ref{ssc:proof-main}. Since the projection of $\tilde\IP$ onto $\Omega_0$ is $\IP_0$, and that of $\IP_0$ onto $\Omega$ is $\IP$, every event determined by the original field $(X_i)_{i\in\Z^d}$ maintains its original probability measure. 
\section*{AI usage statement\texorpdfstring{\protect\footnotemark}{}}\footnotetext{\href{https://leidendeclaration.ai/}{https://leidendeclaration.ai/}}
The authors acknowledge the use of ChatGPT-Pro and Claude Opus in preparing this paper. Based on concrete mathematical arguments provided by the authors, ChatGPT-Pro produced an initial version of the proof, which the authors subsequently modified and simplified significantly. Through repeated interactions, the models were also used to sharpen some of the auxiliary arguments and to polish the exposition. In particular, ChatGPT helped the authors locate the earlier results of \cite{moricz1977moment} and the troika of works \cite{rio1993a}, \cite{rio1993b} and \cite{rio1997}. The authors take full responsibility for the veracity of all claims and technical details in the paper.

\renewcommand{\bibfont}{\small}
\bibliographystyle{plainnat}
\bibliography{references}

@article{strassen1964invariance,
  author = {Strassen, Volker},
  title = {An invariance principle for the law of the iterated logarithm},
  journal = {Zeitschrift f{\"u}r Wahrscheinlichkeitstheorie und Verwandte Gebiete},
  volume = {3}, pages = {211--226}, year = {1964},
}

@article{kmt,
  author = {Koml{\'o}s, J{\'a}nos and Major, P{\'e}ter and Tusn{\'a}dy, G{\'a}bor},
  title = {An approximation of partial sums of independent {RV}'-s, and the sample {DF}. {I}},
  journal = {Zeitschrift f{\"u}r Wahrscheinlichkeitstheorie und Verwandte Gebiete},
  volume = {32}, pages = {111--131}, year = {1975},
}

@article{kmt1976,
  author = {Koml{\'o}s, J{\'a}nos and Major, P{\'e}ter and Tusn{\'a}dy, G{\'a}bor},
  title = {An approximation of partial sums of independent {RV}'s, and the sample {DF}. {II}},
  journal = {Zeitschrift f{\"u}r Wahrscheinlichkeitstheorie und Verwandte Gebiete},
  volume = {34}, pages = {33--58}, year = {1976},
}

@book{csorgo1981strong,
  author = {Cs{\"o}rg{\H o}, Mikl{\'o}s and R{\'e}v{\'e}sz, P{\'a}l},
  title = {Strong Approximations in Probability and Statistics},
  publisher = {Academic Press},
  address = {New York}, year = {1981},
  isbn = {9780121985400},
}

@book{philipp1975almost,
  author = {Philipp, Walter and Stout, William},
  title = {Almost Sure Invariance Principles for Partial Sums of Weakly Dependent Random Variables},
  series = {Memoirs of the American Mathematical Society},
  volume = {2}, note = {No. 161},
  publisher = {American Mathematical Society},
  address = {Providence, RI}, year = {1975},
}

@article{berkes1979approximation,
  author = {Berkes, Istv{\'a}n and Philipp, Walter},
  title = {Approximation theorems for independent and weakly dependent random vectors},
  journal = {The Annals of Probability},
  year = {1979}, volume = {7}, number = {1}, pages = {29--54},
}

@article{wu2005nonlinear,
  author = {Wu, Wei Biao},
  title = {Nonlinear system theory: Another look at dependence},
  journal = {Proceedings of the National Academy of Sciences},
  year = {2005}, volume = {102}, number = {40}, pages = {14150--14154},
}

@article{wu2007strong,
  author = {Wu, Wei Biao},
  title = {Strong invariance principles for dependent random variables},
  journal = {The Annals of Probability},
  year = {2007}, volume = {35}, number = {6}, pages = {2294--2320},
}

@article{liu2009strong,
  author = {Liu, Weidong and Lin, Zhengyan},
  title = {Strong approximation for a class of stationary processes},
  journal = {Stochastic Processes and their Applications},
  volume = {119}, number = {1}, pages = {249--280}, year = {2009},
}

@article{berkes2014kmt,
  author = {Berkes, Istv{\'a}n and Liu, Weidong and Wu, Wei Biao},
  title = {{Koml{\'o}s--Major--Tusn{\'a}dy} approximation under dependence},
  journal = {The Annals of Probability},
  year = {2014}, volume = {42}, number = {2}, pages = {794--817},
}

@article{wuzhou2011gaussian,
  author = {Wu, Wei Biao and Zhou, Zhou},
  title = {{Gaussian} approximations for non-stationary multiple time series},
  journal = {Statistica Sinica},
  year = {2011},
  volume = {21},
  number = {3},
  pages = {1397--1413},
}

@article{karmakar2020optimal,
  author = {Karmakar, Sayar and Wu, Wei Biao},
  title = {Optimal {Gaussian} approximation for multiple time series},
  journal = {Statistica Sinica},
  year = {2020}, volume = {30}, number = {3}, pages = {1399--1417},
}

@article{bonnerjee2024gaussian,
  author = {Bonnerjee, Soham and Karmakar, Sayar and Wu, Wei Biao},
  title = {{Gaussian} approximation for nonstationary time series with optimal rate and explicit construction},
  journal = {The Annals of Statistics},
  year = {2024}, volume = {52}, number = {5}, pages = {2293--2317},
}

@misc{bonnerjee2026fast,
  author = {Bonnerjee, Soham and Karmakar, Sayar and Michailidis, George},
  title = {Fast localization of anomalous patches in spatial data under dependence},
  year = {2026},
  howpublished = {arXiv:2603.27546},
  eprint = {2603.27546}, archivePrefix = {arXiv}, primaryClass = {stat.ME},
}

@article{morrow1986invariance,
  author = {Morrow, Gregory J. and Philipp, Walter},
  title = {Invariance principles for partial sum processes and empirical processes indexed by sets},
  journal = {Probability Theory and Related Fields},
  year = {1986}, volume = {73}, number = {1}, pages = {11--42},
}

@article{rio1993a,
  author = {Rio, Emmanuel},
  title = {Strong approximation for set-indexed partial sum processes via {KMT} constructions {I}},
  journal = {The Annals of Probability},
  year = {1993}, volume = {21}, number = {2}, pages = {759--790},
}

@article{rio1993b,
  author = {Rio, Emmanuel},
  title = {Strong approximation for set-indexed partial-sum processes, via {KMT} constructions {II}},
  journal = {The Annals of Probability},
  year = {1993}, volume = {21}, number = {3}, pages = {1706--1727},
}

@article{rio1997,
  author = {Rio, Emmanuel},
  title = {Strong approximation for set-indexed partial sum processes via {KMT} constructions {III}},
  journal = {ESAIM: Probability and Statistics},
  year = {1997}, volume = {1}, pages = {319--338},
}

@article{berkes1981mixing,
  author = {Berkes, Istv{\'a}n and Morrow, Gregory J.},
  title = {Strong invariance principles for mixing random fields},
  journal = {Zeitschrift f{\"u}r Wahrscheinlichkeitstheorie und Verwandte Gebiete},
  year = {1981}, volume = {57}, number = {1}, pages = {15--37},
}

@article{berkes1984gaussian,
  author = {Berkes, Istv{\'a}n},
  title = {{Gaussian} approximation of mixing random fields},
  journal = {Acta Mathematica Hungarica},
  year = {1984}, volume = {43}, number = {1--2}, pages = {153--185},
}

@incollection{strittmatter1990strong,
  author = {Strittmatter, Wolfram},
  title = {Strong approximations for set-indexed partial-sum processes and empirical processes of mixing random fields},
  booktitle = {Probability in Banach Spaces 7},
  editor = {Eberlein, Ernst and Kuelbs, James and Marcus, Michael B.},
  series = {Progress in Probability}, volume = {21}, pages = {201--268},
  publisher = {Birkh{\"a}user}, address = {Boston}, year = {1990},
}

@article{balan2005strong,
  author = {Balan, Raluca M.},
  title = {A strong invariance principle for associated random fields},
  journal = {The Annals of Probability},
  year = {2005}, volume = {33}, number = {2}, pages = {823--840},
}

@article{shashkin2007strong,
  author = {Shashkin, Alexey P.},
  title = {Strong {Gaussian} approximation of an associated random field},
  journal = {Russian Mathematical Surveys},
  year = {2007}, volume = {62}, number = {5}, pages = {1012--1014},
}

@incollection{bulinski2006strong,
  author = {Bulinski, Alexander and Shashkin, Alexey},
  title = {Strong invariance principle for dependent random fields},
  booktitle = {Dynamics \& Stochastics},
  series = {Institute of Mathematical Statistics Lecture Notes--Monograph Series},
  volume = {48}, pages = {128--143}, year = {2006},
  publisher = {Institute of Mathematical Statistics},
}

@article{cai2011strong,
  author = {Cai, Guang-hui},
  title = {A strong invariance principle for negatively associated random fields},
  journal = {Czechoslovak Mathematical Journal},
  year = {2011}, volume = {61}, number = {1}, pages = {27--40},
}

@article{bolthausen1982central,
  author = {Bolthausen, Erwin},
  title = {On the central limit theorem for stationary mixing random fields},
  journal = {The Annals of Probability},
  year = {1982}, volume = {10}, number = {4}, pages = {1047--1050},
}

@article{dedecker1998central,
  author = {Dedecker, J{\'e}r{\^o}me},
  title = {A central limit theorem for stationary random fields},
  journal = {Probability Theory and Related Fields},
  year = {1998}, volume = {110}, number = {3}, pages = {397--426},
}

@article{dedecker2001exponential,
  author = {Dedecker, J{\'e}r{\^o}me},
  title = {Exponential inequalities and functional central limit theorems for random fields},
  journal = {ESAIM: Probability and Statistics},
  year = {2001}, volume = {5}, pages = {77--104},
}

@article{wang2013new,
  author = {Wang, Yizao and Woodroofe, Michael},
  title = {A new condition for the invariance principle for stationary random fields},
  journal = {Statistica Sinica},
  year = {2013}, volume = {23}, number = {4}, pages = {1673--1696},
}

@article{el2013central,
  author = {El Machkouri, Mohamed and Voln{\'y}, Dalibor and Wu, Wei Biao},
  title = {A central limit theorem for stationary random fields},
  journal = {Stochastic Processes and their Applications},
  year = {2013}, volume = {123}, number = {1}, pages = {1--14},
}

@article{volny2014hannan,
  author = {Voln{\'y}, Dalibor and Wang, Yizao},
  title = {An invariance principle for stationary random fields under {Hannan}'s condition},
  journal = {Stochastic Processes and their Applications},
  year = {2014}, volume = {124}, number = {12}, pages = {4012--4029},
}

@article{giraudo2018invariance,
  author = {Giraudo, Davide},
  title = {Invariance principle via orthomartingale approximation},
  journal = {Stochastics and Dynamics},
  year = {2018}, volume = {18}, number = {6}, pages = {1850043},
}

@article{cuny2026weak,
  author = {Cuny, Christophe and Dedecker, J{\'e}r{\^o}me and Merlev{\`e}de, Florence},
  title = {On the weak invariance principle for random fields with commuting filtrations under {$\mathbb{L}^{1}$}-projective criteria},
  journal = {Electronic Journal of Probability},
  year = {2026}, volume = {31}, pages = {1--37}, note = {Paper No. 1},
}

@article{volny2017erratum,
  author = {Voln{\'y}, Dalibor and Wang, Yizao},
  title = {Erratum to {An invariance principle for stationary random fields under Hannan's condition} [{Stochastic Process. Appl.} 124(12) (2014) 4012--4029]},
  journal = {Stochastic Processes and their Applications},
  year = {2017}, volume = {127}, number = {6}, pages = {2088--2091},
}

@article{moricz1977moment,
  author = {M{\'o}ricz, Ferenc},
  title = {Moment inequalities for the maximum of partial sums of random fields},
  journal = {Acta Scientiarum Mathematicarum},
  year = {1977}, volume = {39}, number = {3--4}, pages = {353--366},
}

@article{cattaneo2025yurinskii,
  author = {Cattaneo, Matias D. and Masini, Ricardo P. and Underwood, William G.},
  title = {{Yurinskii}'s coupling for martingales},
  journal = {The Annals of Statistics},
  year = {2025}, volume = {53}, number = {5}, pages = {2179--2203},
}

@incollection{burkholder1988sharp,
  author    = {Burkholder, Donald L.},
  title     = {Sharp inequalities for martingales and stochastic integrals},
  booktitle = {Colloque Paul L\'evy sur les processus stochastiques},
  series    = {Ast\'erisque},
  number    = {157--158},
  pages     = {75--94},
  year      = {1988},
  publisher = {Soci\'et\'e math\'ematique de France},
}

@book{dudley1999uniform,
  author = {Dudley, Richard M.},
  title = {Uniform Central Limit Theorems},
  series = {Cambridge Studies in Advanced Mathematics}, volume = {63},
  publisher = {Cambridge University Press}, address = {Cambridge}, year = {1999}}

@article{vorobev1962consistent,
  author = {Vorob'ev, Nikolai N.},
  title = {Consistent families of measures and their extensions},
  journal = {Theory of Probability and its Applications}, volume = {7}, pages = {147--163}, year = {1962}}

@book{kallenberg2002foundations,
  author = {Kallenberg, Olav},
  title = {Foundations of Modern Probability},
  edition = {Second}, series = {Probability and its Applications},
  publisher = {Springer}, address = {New York}, year = {2002}}

@book{ibragimov1978gaussian,
  author = {Ibragimov, Il'dar A. and Rozanov, Yurii A.},
  title = {Gaussian Random Processes},
  series = {Applications of Mathematics}, volume = {9},
  publisher = {Springer}, address = {New York}, year = {1978}}

@article{tropp2012user,
  title={User-friendly tail bounds for sums of random matrices},
  author={Tropp, Joel A},
  journal={Foundations of computational mathematics},
  volume={12},
  number={4},
  pages={389--434},
  year={2012},
  publisher={Springer}
}

@article{Moricz1979MultiparameterII,
  author  = {M{\'o}ricz, Ferenc},
  title   = {Multiparameter strong laws of large numbers. {II}
             (Higher order moment restrictions)},
  journal = {Acta Scientiarum Mathematicarum (Szeged)},
  volume  = {41},
  year    = {1979},
  pages   = {339--349}
}

@article{Moricz1983Rectangular,
  author  = {M{\'o}ricz, Ferenc},
  title   = {A general moment inequality for the maximum of the rectangular
             partial sums of multiple series},
  journal = {Acta Mathematica Hungarica},
  volume  = {41},
  number  = {3--4},
  year    = {1983},
  pages   = {337--346},
}

@article{christofides1990maximal,
  title={Maximal inequalities for multidimensionally indexed submartingale arrays},
  author={Christofides, Tasos C and Serfling, Robert J},
  journal={The Annals of Probability},
  volume={18},
  number={2},
  pages={630--641},
  year={1990},
  publisher={Institute of Mathematical Statistics}
}

@article{bretagnolle1989hungarian,
  title={Hungarian constructions from the nonasymptotic viewpoint},
  author={Bretagnolle, Jean and Massart, Pascal},
  journal={The Annals of Probability},
  volume={17},
  number={1},
  pages={239--256},
  year={1989},
  publisher={Institute of Mathematical Statistics}
}

@article{hirsch1995potential,
  title={Potential theory related to some multiparameter processes},
  author={Hirsch, Francis},
  journal={Potential Analysis},
  volume={4},
  number={3},
  pages={245--267},
  year={1995},
  publisher={Springer}
}

@incollection{walsh2006martingales,
  title={Martingales with a multidimensional parameter and stochastic integrals in the plane},
  author={Walsh, John B},
  booktitle={Lectures in Probability and Statistics: Lectures given at the Winter School in Probability and Statistics held in Santiago de Chile},
  pages={329--491},
  year={1986},
  publisher={Springer}
}

@article{Cairoli1970,
author = {Cairoli, Renzo},
journal = {Séminaire de probabilités de Strasbourg},
language = {fre},
pages = {1-27},
publisher = {Springer - Lecture Notes in Mathematics},
title = {Une inégalité pour martingales à indices multiples et ses applications},
volume = {4},
year = {1970},
}
\newpage
\appendix
\section*{Organization of the appendix}
This appendix contains the deferred proofs and other auxiliary results. In particular, Appendix~\ref{sec:proof-GA} proves the local Gaussian approximation in Theorem~\ref{thm:GA}.  Moving on, Appendix~\ref{sec:proof-spectral} proves the spectral comparison in Theorem~\ref{thm:spectral}, by realizing the stationary Gaussian field and the white noise through a common spectral white noise. Subsequently, Appendix~\ref{sec:helper-results} collects the auxiliary results, modularized according to their use: \S\ref{ssc:aux-main}, \S\ref{ssc:aux-GA} and \S\ref{ssc:aux-spectral} contain those used in the proofs of Theorems~\ref{thm:main}, \ref{thm:GA} and \ref{thm:spectral}, respectively, while \S\ref{ssc:aux-additional} contains moment and maximal inequalities used throughout. Finally, Appendix~\ref{sec:helper-proofs} contains the proofs of the auxiliary results.

\section{Proof of Theorem \ref{thm:GA}}\label{sec:proof-GA}
We prove Theorem~\ref{thm:GA} in the following steps.
\begin{itemize}
        \item \textbf{Haar coordinates} (\S\ref{subsec:haar}): as in the dyadic constructions of \cite{bretagnolle1989hungarian, rio1997}, we work with the scaled tensor Haar transform $\HA$ on $\R^Q$, whose coordinates $(\HA x)_\lambda=\sum_{i\in Q_\lambda}\pm x_i$ are signed sums of $x$ over the dyadic sub-rectangles $Q_\lambda$ of $Q$. Lemma \ref{lem:Haar-algebra} yields
    \[
    \max_{m\in Q}\Big|\sum_{i\in[1,m]}(x_i-g_i)\Big|\le(k+1)^d\,|\HA(x-g)|_\infty, \  x,g\in\R^Q,
    \]
    so that $\ell^\infty$ control in Haar coordinates yields control of all lower-left rectangular sums.
    \item \textbf{$m$-dependent approximation} (\S\ref{ssc:mdep}): $X_i$ is replaced by its $2^k$-dependent approximation $X_i^{[2^k]}$, at a polynomial cost which is negligible compared to the threshold in Theorem \ref{thm:GA}. By the preceding step, it then suffices to couple $Z:=\HA(X_i^{[2^k]})_{i\in Q}$ with a Gaussian vector in $\ell^\infty$.
        \item \textbf{Dyadic projections} (\S\ref{ssc:scales}--\S\ref{ssc:bits}): telescoping the 
        $m$-dependent approximations along dyadic scales gives
    \[
    X_i^{[2^k]}=\sum_{s=0}^k\xi_i^{(s)}, \  \xi_i^{(0)}=X_i^{[1]}, \  \xi_i^{(s)}=X_i^{[2^s]}-X_i^{[2^{s-1}]},
    \]
    where, for $s\ge1$, $\xi_i^{(s)}$ are the orthogonal projection of $X_i$. Each $(\xi_i^{(s)})_{i}$ is thus a $2^s$-dependent field, whose functional dependence is $\lesssim2^{-sA}$ by \eqref{eq:residual-rate-main}. To write $Z$ as a martingale, we reveal the innovations one binary digit at a time, via $\eta_h=F_\eta^{-1}(\sum_{b\ge1}2^{-b}B_{h,b})$ with i.i.d.\ fair bits $B_{h,b}$. This lets us write
    \[
    \IE[Z\mid\cG_{h,b}]-\IE[Z\mid\cG_{h,b-1}]=(2B_{h,b}-1)\,v_{h,b},
    \]
    where $\cG_{h,b}$ is generated by the bits revealed up to $(h,b)$. We call the corresponding Doob martingale along this filtration as the \textit{bitwise martingale}. Note that  all conditional third moments vanish.

            \item \textbf{Coupling} (\S\ref{ssc:beta}--\S\ref{ssc:coupling}): the bitwise martingale facilitates invoking the third-order Yurinskii coupling of \cite{cattaneo2025yurinskii}, subsequently, we control the terms resulting from this application.
\end{itemize}
Now we proceed towards a formal proof. Fix $k\ge1$ and $x\in(0,1/2)$, recall $Q$, $r$ and $\nu_A$ from Theorem \ref{thm:GA}, and let $c:=d(1-2/r)$, so that $\nu_A=1+\sum_{\ell=1}^k2^{\ell(c-A)}$. Here and below, $\lesssim$ hides constants depending only on $d$, $p$, $A$, $\Delta_p$ and the implied constant in \eqref{eq:residual-rate-main}. We denote the right hand side in the large deviation bound of Theorem \ref{thm:CMU} as
\begin{equation}\nonumber
    \mathfrak{b}\asymp x^{-3/2}\,2^{kd/r}\,\nu_A^{1/2}\,(k+1)^{3d/2+1-1/r}.
\end{equation}
Further, denote $\tilde\BX^Q:=(X_i)_{i\in Q}$. With the notational bobbery dealt with, we provide each of the key steps in this proof.

\subsection{Preliminaries: Haar basis and Haar transform} \label{subsec:haar}
Similar to \cite{bretagnolle1989hungarian, rio1997}, we proceed via introducing the Haar basis of the indicator functions of intervals, and relating it to the cubes in our problem.  A dyadic interval in $[2^k]$ is an interval obtained by repeatedly bisecting the whole set.  If $I$ is a dyadic interval with $\abs I\ge2$, let $I^-$ and $I^+$ be its left and right halves and define
\begin{equation*}
  b_I=\abs I^{-1/2}(\ind_{I^-}-\ind_{I^+}).
\end{equation*}
Together with the father function $\phi=2^{-k/2}\ind_{[2^k]},$ these functions form an orthonormal basis $\cB_k$ of $\R^{[2^k]}$.  For $b\in\cB_k$, write $I_b$ for its support, with $I_\phi=[2^k]$.  Then
\begin{equation}\label{eq:haar-values}
  b(i)\in\{0,\pm\abs{I_b}^{-1/2}\}.
\end{equation}
For a tuple $\lambda=(b_1,\dots,b_d)\in\cB_k^d,$
define its corresponding tensor-product function $\psi_\lambda: Q \mapsto \R$, given by
\begin{equation}\label{eq:tensor-haar}
  \psi_\lambda(i)=\prod_{j=1}^d b_j(i_j),\ i=(i_1, \ldots, i_d)\in Q,
  \ \text{and its support rectangle }
  Q_\lambda=\prod_{j=1}^d I_{b_j}.
\end{equation}
In other words, $\psi_{\lambda}(i)=0$ if $i \notin Q_{\lambda} \subseteq Q$. Moreover, writing out $\psi_{\lambda}:=\{\psi_{\lambda}(i); i\in Q\}$ as a vector in $\R^Q$, it follows immediately that $\{\psi_{\lambda}\}_{\lambda \in \cB_k^d}$ constitutes an orthonormal basis of $\R^Q$.  Let $\HO:= (\HO_{\lambda, i})_{\lambda\in \cB_k^d, i \in Q}$ be the orthogonal $2^{kd}\times 2^{kd}$ matrix whose $\lambda$th row is $\psi_\lambda$.  Put
\begin{align}
    \HD=\operatorname{diag}(\abs{Q_\lambda}^{1/2}:\lambda\in\cB_k^d),
  \ \text{and }
  \HA=\HD\HO. \nonumber
\end{align}
The following result collects some technical tools pertinent to the treatment and application of the Haar bases.
\begin{lemma}\label{lem:Haar-algebra}
It follows that
\begin{enumerate}[label=\textup{(\alph*)}]
\item $\HA_{\lambda i}\in\{0,\pm1\}$;
\item $\rho^\star(\HA)^2=2^{kd}$;
\item every column of $\HA$ has exactly $(k+1)^d$ nonzero entries, and $|\HA_{\cdot i}|=(k+1)^{d/2}$;
\item for every $m\in Q$,
\begin{equation*}
  \sum_{\lambda\in\cB_k^d}
  \frac{\abs{\inner{\ind_{[1,m]}, {\psi_\lambda}}}}
       {\abs{Q_\lambda}^{1/2}}
  \le (k+1)^d;
\end{equation*}
\item for all $x,g\in\R^Q$, $\max_{m\in Q}
  \big|\sum_{i\in[1,m]}(x_i-g_i)\big|
  \le (k+1)^d|{\HA(x-g)}|_\infty.$
\end{enumerate}
\end{lemma}
Two consequences of Lemma \ref{lem:Haar-algebra} are used repeatedly. Firstly, applying (e) with $x=\HA^{-1}y$ and $g=0$,
\begin{equation}\label{eq:haar-inverse}
    \max_{m\in Q}\Big|\sum_{i\in[1,m]}(\HA^{-1}y)_i\Big|\le (k+1)^d|y|_\infty,\  y\in\R^Q.
\end{equation}
Secondly, by (c) and Minkowski's inequality, for any random variables $(w_i)_{i\in Q}$,
\begin{equation}\label{eq:haar-column}
    \big\||\HA w|\big\|_p\le\sum_{i\in Q}|\HA_{\cdot i}|\,\|w_i\|_p=(k+1)^{d/2}\sum_{i\in Q}\|w_i\|_p,\  w=(w_i)_{i\in Q}.
\end{equation}

\subsection{$m$-dependent approximation and reduction to Haar coordinates}\label{ssc:mdep}
Recall $X_i^{[m]}$ as in \eqref{eq:local-def}. Lemma \ref{lem:maximal}, applied to the residual field $(X_i-X_i^{[2^k]})_{i}$, along with \eqref{eq:residual-rate-main} instructs that
\begin{align}
    \Big\|\max_{m \in Q}\big| \sum_{i\in [1,m]}(X_i - X_i^{[2^k]}) \big| \Big\|_p \lesssim \Delta_p^{[2^k]} 2^{kd/2}\lesssim 2^{k(d/2-A)}.\label{eq:ga-m-dep}
\end{align}
The polynomial cost in \eqref{eq:ga-m-dep} is dominated by the polynomial part of $\mathfrak{b}$: since $c=d-2d/r$ and $\nu_A\ge\max\{1,2^{k(c-A)}\}$,
\begin{equation}\label{eq:mdep-vs-nu}
    2^{k(d/2-A)}\le 2^{k(d-A)/2}\le 2^{k\max\{d/r,\,(d-A)/2\}}\le 2^{kd/r}\nu_A^{1/2}.
\end{equation}

Write $\BX^{[2^k]}=(X_i^{[2^k]})_{i\in Q}$ and $Z:=\HA\BX^{[2^k]}\in\R^Q$, where the dependence on $Q$ is suppressed for notational convenience. The Gaussian approximation then aims to construct a Gaussian vector $\bG^Q=(G^Q_i)_{i\in Q}$ such that all lower-left rectangular sums of $\BX^{[2^k]}-\bG^Q$ are small with high probability. In fact, we construct a Gaussian vector $T_G$ such that $|Z - T_G|_{\infty}$ is small, after which the linear transformation $\bG^Q=\HA^{-1}T_G$ does the job: by \eqref{eq:haar-inverse},
\begin{align}
    \max_{m \in Q}\big| \sum_{i\in [1,m]}(X_i^{[2^k]}- (\HA^{-1}T_G)_i) \big| \le (k+1)^d |Z - T_G|_{\infty}. \label{eq:ga-haar-appl}
\end{align}
Conveniently, in light of Lemma \ref{lem:Haar-algebra}, replacing simultaneous control of all lower-left rectangular sums by control of the $\ell^\infty$ norm in Haar coordinates worsens the approximation rate only by the logarithmic factor $(k+1)^d$. Moreover, such Gaussian approximations controlling the $\ell^{\infty}$ norm are readily available from \cite{cattaneo2025yurinskii}, once $Z$ is written as a sum of martingale differences with suitable conditional structure; this is accomplished in \S\ref{ssc:scales}--\S\ref{ssc:bits}.

\subsection{Decomposition into local scales}\label{ssc:scales}
We proceed via decomposing $X_i^{[2^k]}$ into projection-based increments across dyadic radii as follows:
\begin{equation}\nonumber
    \xi_i^{(0)}=X_i^{[1]}, \ \text{and, for } s\in [k], \ \xi_i^{(s)}= X_i^{[2^s]}- X_i^{[2^{s-1}]},\  \text{so that}\   X_i^{[2^k]}=\sum_{s=0}^k\xi^{(s)}_i.
\end{equation}
By construction,
\begin{equation}\label{eq:xi-local}
    \xi^{(s)}_i \text{ is a function of } (\eta_{i-u})_{|u|_\infty\le 2^s}.
\end{equation}
Let $\xi_i^{(s),*}$ be obtained by replacing $\eta_0$ by an i.i.d. copy $\eta_0'$, and set $ \delta_{j,p}^{(s)}=\norm{\xi_j^{(s)}-\xi_j^{(s),*}}_p$ and $\Theta_{s,p}=\sum_{j\in\Z^d}\delta_{j,p}^{(s)}.$ For $s\ge1$, rewriting $\xi_i^{(s)}= (X_i -X_i^{[2^{s-1}]})- (X_i -X_i^{[2^{s}]})$, it follows from the triangle inequality and \eqref{eq:residual-rate-main} that $$\Theta_{s,p} \le \Delta_p^{[2^{s-1}]}+ \Delta_p^{[2^s]}\lesssim 2^{-sA}.$$ For $s=0$, if $|i|_\infty\le1$ then $$X^{[1]}_i-X^{[1],*}_i=\IE[X_i-X^\star_i\mid\eta_{i-u},|u|_\infty\le1;\eta'_0],$$ while the difference vanishes if $|i|_\infty>1$; hence the conditional Jensen's inequality gives $\Theta_{0,p}\le\Delta_p$. Consequently,
  \begin{align}
      \Theta_{s,p} \lesssim 2^{-sA},\ 0\le s\le k,\ \text{and}\ \sum_{s=0}^k\Theta_{s,p}\lesssim 1. \label{eq:bound-on-theta}
  \end{align}
  For $0\le s\le k$, denote $\Xi^{(s)}:= (\xi^{(s)}_i)_{i \in Q}$ and $Z^{(s)}= \HA {\Xi}^{(s)}$, so that $Z= \sum_{s=0}^k Z^{(s)}$. Here, since $Q$ is finite, we assume that the entries of $Q$ have been enumerated in the same manner for $\HA$, $\Xi^{(s)}$, and in all subsequent quantities to which $\HA$ will be multiplied. By \eqref{eq:xi-local}, $Z$ depends only on the innovations indexed by
  \begin{equation}\nonumber
      H_Q:= \{h \in \Z^d: \dist_\infty(h,Q)\le 2^k \},\  |H_Q|=3^d2^{kd}.
  \end{equation}

\subsection{The \textit{bitwise martingale}}\label{ssc:bits}
By \S\ref{ssc:scales}, $Z$ is a function of the finitely many innovations $(\eta_h)_{h\in H_Q}$. To apply the coupling of \cite{cattaneo2025yurinskii}, we write $Z$ as a sum of martingale differences with vanishing conditional third moments. Revealing the innovations one site $i\in \Z^d$ at a time does not achieve this, since the resulting increments are general functions of $\eta_h$, whose conditional third moments need not vanish. We therefore reveal them one binary digit at a time. Recall that $\eta_h=F_\eta^{-1}(U_h)$ with $(U_h)_{h\in\Z^d}$ i.i.d.\ $\mathrm{Unif}(0,1)$. Outside a null set, $U_h=\sum_{b\ge1}2^{-b}B_{h,b}$, where $\{B_{h,b}\}_{h\in\Z^d,b\ge1}$ are i.i.d.\ $\mathrm{Ber}(1/2)$ random variables; thus $\eta_h = F^{-1}_{\eta}(\sum_{b=1}^{\infty}2^{-b}B_{h,b})$. Since the number of Bernoulli variables $\{B_{h,b}\}$ is countably infinite, whereas the martingale coupling of \cite{cattaneo2025yurinskii} concerns finitely many martingale differences, we reveal only the first $B$ bits at every site of $H_Q$, where $B\in\N$ is fixed but large, to be chosen later (in \eqref{eq:var-3} and \S\ref{ssc:4errors} (b)). Enumerating $H_Q=\{h_1, \ldots, h_H\}$, let
\begin{equation}\nonumber
    \cG_{h_j, b}=\sigma(\{B_{h_i, c}:i\in [j-1], c\in [B]\}\cup \{B_{h_j, c}: c\in [b]\} ),\  j\in[H],\ b\in[B],
\end{equation}
 where by convention we take $\cG_{h_j, 0}=\cG_{h_{j-1}, B}$ and $\cG_{h_1,0}=\{\emptyset,\Omega\}$. Ordered lexicographically in $(j,b)$, these form a filtration of finite $\sigma$-fields; let $\cB_B:=\cG_{h_H,B}=\sigma(B_{h,b}: h \in H_Q, b\in [B] )$, and put
 \[Z_B:=\IE[Z\mid\cB_B],\  Z^{(s)}_B:=\IE[Z^{(s)}\mid\cB_B].\]
 Since $Z\in L^p$ is measurable with respect to $\sigma(B_{h,b}:h\in H_Q,b\ge1)=\sigma(\cup_B\cB_B)$, martingale convergence entails
 \begin{equation}\label{eq:ZB-conv}
     Z_B\to Z\  \text{almost surely and in } L^p,\  \text{as } B\to\infty.
 \end{equation}
Let
 \[ D_{h,b}:= \IE[Z|\cG_{h,b}] - \IE[Z | \cG_{h, b-1}], \ h\in H_Q, b\in [B], \]
 and likewise define $D_{h, b}^{(s)}$ with $Z^{(s)}$ in place of $Z$, so that $D_{h,b}=\sum_{s=0}^k D_{h,b}^{(s)}$. Since $\IE Z=0$, $\sum_{j=1}^H \sum_{b=1}^B D_{h_j, b}= Z_B$. Finally, for $\eps\in \{0,1\}$, let
 \[z_{\eps, h, b}:= 2\IE[Z\ind\{B_{h,b}=\eps\} | \cG_{h, b-1}], \  z_{\eps, h, b}^{(s)}:= 2\IE[Z^{(s)}\ind\{B_{h,b}=\eps\} | \cG_{h, b-1}],\]
 which are $\cG_{h,b-1}$-measurable. Since $B_{h,b}$ is a fair bit independent of $\cG_{h,b-1}$, we can write
\begin{align}
    \IE[Z | \cG_{h,b}]&= B_{h,b}z_{1, h, b} + (1-B_{h,b})z_{0,h,b} \ \text{and}, \ \IE[Z | \cG_{h,b-1}] =\frac{z_{1,h,b}+z_{0, h, b}}{2} \nonumber,
\end{align}
which implies that
\begin{align}
    D_{h,b}= \zeta_{h,b}v_{h,b},\   D_{h,b}^{(s)}= \zeta_{h,b}v^{(s)}_{h,b},\  \text{where}\  & \zeta_{h,b}:=2B_{h,b}-1,\nonumber\\ v_{h,b}:=\frac{z_{1, h, b}-z_{0,h,b}}{2},\  & v^{(s)}_{h,b}:=\frac{z^{(s)}_{1, h, b}-z^{(s)}_{0,h,b}}{2}. \label{eq:D_hb}
\end{align}
In particular, $v_{h,b}=\sum_{s=0}^kv^{(s)}_{h,b}$. The following lemma records the conditional structure that makes the third-order coupling applicable.
\begin{lemma}\label{lem:rademacher}
The increments $D_{h,b}$ in \eqref{eq:D_hb} satisfy the following.
\begin{enumerate}[label=\textup{(\roman*)}, leftmargin=*]
    \item $\{D_{h_j,b}\}$ is a martingale difference sequence adapted to the filtration $\{\cG_{h_j,b}\}$ ordered lexicographically, with $\sum_{j,b}D_{h_j,b}=Z_B$.
    \item $|D_{h,b}|=|v_{h,b}|$ is $\cG_{h,b-1}$-measurable.
    \item $V_{h,b}:=\IE[D_{h,b}D_{h,b}^\top\mid\cG_{h,b-1}]=v_{h,b}v_{h,b}^\top=D_{h,b}D_{h,b}^\top$.
    \item All conditional third moments vanish: $\IE[D_{h,b,\lambda_1}D_{h,b,\lambda_2}D_{h,b,\lambda_3}\mid\cG_{h,b-1}]=0$ for all coordinates $\lambda_1,\lambda_2,\lambda_3$.
\end{enumerate}
\end{lemma}
Finally in order to apply, Corollary 2.2 of \cite{cattaneo2025yurinskii} (Theorem \ref{thm:CMU} below) to the martingale difference sequence $\{D_{h_j,b}\}$, we define the following terms. Let $\{W_{h_j,b}\}_{j\in [H], b \in [B]}$ be i.i.d. $N(0, I_Q),$ independent of all the random variables above, recall $V_{h,b}$ from Lemma \ref{lem:rademacher}(iii), and define
\begin{align}
    \beta_{\infty, 3}&:= \sum_{h, b}\IE\Big[|D_{h,b}|^3 |D_{h,b}|_{\infty} + |V_{h,b}^{1/2} W_{h,b}|^3|V_{h,b}^{1/2} W_{h,b}|_{\infty} \Big], \ \text{and,}\nonumber\\
    \Omega &= \sum_{h,b} V_{h,b} - \Sigma, \ \Sigma= \HA \Sigma_Q \HA^\top,\nonumber
\end{align}
where $\Sigma_Q=\Cov((X_i)_{i\in Q})$. In the following, we attempt to control $\beta_{\infty, 3}$.

\subsection{\texorpdfstring{Third-moment bounds and truncation\color{black}}{Third-moment bounds and truncation}}\label{ssc:fail}
Recall $M_h^{(s)}$ from Lemma \ref{lem:M_h}, and let $M_h:=(\sum_b|D_{h,b}|^2)^{1/2}$. Triangle inequality gives $M_h\le\sum_sM^{(s)}_h$, so that
\begin{align}
    \sum_{h,b}|D_{h,b}|^3 |D_{h,b}|_{\infty} \le \sum_{h,b} |D_{h,b}|^4 \le \sum_h \big(\sum_b |D_{h,b}|^2 \big)^2 =\sum_hM_h^4 &\le \sum_h \big(\sum_s M_h^{(s)}\big)^4 \nonumber.
\end{align}
On the other hand, a version of Burkholder's inequality (Lemma \ref{lem:M_h}) and \eqref{eq:bound-on-theta} jointly imply
\begin{equation}\label{eq:Mh-bound}
    \|M_h\|_p\le \|\sum_s M_h^{(s)}\|_p \lesssim (k+1)^{d/2}\sum_{s=0}^k \Theta_{s,p} \lesssim (k+1)^{d/2},\  h\in H_Q.
\end{equation}
Therefore, if $p\ge 4$, it follows that $ \| M_h\|_4^4 \le \|M_h\|_p^4 \lesssim (k+1)^{2d}$, which provides
\[\IE\Big[\sum_{h,b}|D_{h,b}|^3 |D_{h,b}|_{\infty}\Big] \lesssim 2^{kd} (k+1)^{2d}\]
in light of $|H_Q|= 3^d 2^{kd}$.

However, this step cannot go through when $p\in (2,4)$. Despite $D_{h,b}$ possessing moments of all orders, its moments depend on $B$, which we will take to be potentially very large so that $Z - Z_B$ is a negligible quantity. Therefore, in order to establish sharper control of the moments, we further truncate $D_{h,b}$ when $p\in (2,4)$. Denote $u:=2^{kd/p}(k+1)^{d/2}$, and
\begin{align}
    \wh{D}_{h,b}= \begin{cases}
        D_{h,b}, & p \ge 4,\\
        D_{h,b} \ind{\{ |D_{h,b}|\le u\}}, & 2<p<4.
    \end{cases} \nonumber
\end{align}
Firstly, by Lemma \ref{lem:rademacher}, $\{\wh{D}_{h_j,b}\}_{j\in [H], b\in [B]}$ remains a martingale difference sequence with respect to $\{\cG_{h_j,b}\}_{j\in [H], b \in [B]}$; in fact, $\wh D_{h,b}=\zeta_{h,b}\wh v_{h,b}$ with $\wh v_{h,b}:=v_{h,b}\ind\{|v_{h,b}|\le u\}$ (and $\wh v_{h,b}=v_{h,b}$ if $p\ge4$), and the conditional identities in Lemma \ref{lem:rademacher}(ii)--(iv) apply to $\wh D_{h,b}$ and its conditional covariance. Moreover, since $p/2\ge1$, $\sum_b|D_{h,b}|^p\le(\sum_b|D_{h,b}|^2)^{p/2}=M_h^p$, so that \eqref{eq:Mh-bound} yields
\begin{equation}\label{eq:Dp-sum}
    \sum_{h,b}\IE[|D_{h,b}|^p]\le\sum_h\IE[M_h^p]\lesssim 2^{kd}(k+1)^{pd/2}.
\end{equation}
Consequently, for $2<p<4$,
\begin{align}
    \IE\Big[|\sum_{h,b} (D_{h,b} - \wh{D}_{h,b})|_{\infty}\Big] &\le \sum_{h,b}\IE[|D_{h,b}|\ind{\{ |D_{h,b}|> u\}}]  \le  u^{1-p} \sum_{h,b}\IE[|D_{h,b}|^p]\lesssim (k+1)^{d/2} 2^{kd/p}. \label{eq:trunc-approx}
\end{align}
Equation \eqref{eq:trunc-approx} allows us to focus on establishing a Gaussian approximation for $\wh S:=\sum_{h,b}\wh{D}_{h,b}$. Correspondingly, define $\wh{V}_{h,b}:=\IE[\wh{D}_{h,b}\wh{D}_{h,b}^\top|\cG_{h,b-1}]=\wh v_{h,b}\wh v_{h,b}^\top$, and denote
\begin{align}
    \wh{\beta}_{\infty, 3}&:= \sum_{h, b}\IE\Big[|\wh{D}_{h,b}|^3 |\wh{D}_{h,b}|_{\infty} + |\wh{V}_{h,b}^{1/2} W_{h,b}|^3|\wh{V}_{h,b}^{1/2} W_{h,b}|_{\infty} \Big], \ \text{and,}\nonumber\\
    \wh{\Omega} &= \sum_{h,b} \wh{V}_{h,b} - \Sigma, \ \Sigma= \HA \Sigma_Q \HA^\top.\nonumber
\end{align}
\subsection{Control of \texorpdfstring{$\wh{\beta}_{\infty, 3}$}{beta-hat(infinity, 3)}}\label{ssc:beta}
Note that, similar to the treatment in \S\ref{ssc:fail}, it follows that for $p \ge 4$,
\begin{align}
    \IE\Big[\sum_{h,b}|\wh{D}_{h,b}|^3 |\wh{D}_{h,b}|_{\infty}\Big] \leq \sum_{h,b}\IE[|\wh D_{h,b}|^4]\le \sum_h \| M_h \|_p^4 \lesssim 2^{kd} (k+1)^{2d}. \label{eq:beta-first}
\end{align}
On the other hand, for $p \in (2,4)$, \eqref{eq:Dp-sum} yields
\begin{align}
    \IE\Big[\sum_{h,b}|\wh{D}_{h,b}|^3 |\wh{D}_{h,b}|_{\infty}\Big] \leq \sum_{h,b} \IE[|\wh{D}_{h,b}|^4] &\leq u^{4-p}\sum_{h,b} \IE[|{D}_{h,b}|^p]\nonumber\\
    &\lesssim \big(2^{\frac{kd}{p}} (k+1)^{d/2}\big)^{4-p} 2^{kd} (k+1)^{pd/2}= 2^{4kd/p} (k+1)^{2d}\label{eq:beta-second}.
\end{align}
A similar calculation holds for the second summand of $\wh{\beta}_{\infty, 3}$. Indeed, $\wh V_{h,b}=\wh v_{h,b}\wh v_{h,b}^\top$ has rank at most one, so that, whenever $\wh v_{h,b}\ne0$,
\[\wh V_{h,b}^{1/2}=\wh v_{h,b}\wh v_{h,b}^\top/|\wh v_{h,b}|\]
(and $\wh V_{h,b}^{1/2}=0$ otherwise), and
\[\wh V_{h,b}^{1/2}W_{h,b}=g_{h,b}\,\wh v_{h,b},\  g_{h,b}:=\wh v_{h,b}^\top W_{h,b}/|\wh v_{h,b}|,\]
where, conditional on $\cG_{h,b-1}$, $g_{h,b}\sim N(0,1)$ since $W_{h,b}$ is independent of $\cG_{h,b-1}$. Hence
\[|\wh V_{h,b}^{1/2}W_{h,b}|^3|\wh V_{h,b}^{1/2}W_{h,b}|_\infty\le g_{h,b}^4|\wh v_{h,b}|^4,\]
whose conditional expectation equals $3|\wh v_{h,b}|^4=3|\wh D_{h,b}|^4$. Combining \eqref{eq:beta-first}-\eqref{eq:beta-second} yields
\begin{align}
    \wh{\beta}_{\infty, 3} \le 4\sum_{h,b}\IE[|\wh D_{h,b}|^4]\lesssim 2^{4kd/r} (k+1)^{2d}, \ r=p\wedge 4. \label{eq:beta-final}
\end{align}
\subsection{Control of \texorpdfstring{$\IE[\rho^{\star}(\wh{\Omega})]$}{E[rho-star(Omega-hat)]}}\label{ssc:Omega}
Recall $V_{h,b}=D_{h,b} D_{h,b}^\top$ from Lemma \ref{lem:rademacher}, and let $\Sigma_Q^{[2^k]}=\cov(\BX^{[2^k]})$. Noting that $\Cov(Z)= \HA\Sigma_Q^{[2^k]} \HA^\top$, write
\begin{align}
    \IE[\rho^{\star}(\wh{\Omega})]&\leq \IE\Bigg[\rho^{\star}\Big(\sum_{h,b} (\wh{V}_{h,b} - V_{h,b})\Big)\Bigg] + \IE\Bigg[\rho^\star\Big(\sum_{h,b} V_{h,b} - \Cov(Z_B) \Big) \Bigg] \nonumber\\ &\hspace{2cm}+ \rho^\star\big(\Cov(Z_B) - \Cov(Z) \big) + \rho^\star\big(\HA(\Sigma_Q^{[2^k]}  - \Sigma_Q)\HA^\top\big). \label{eq:var-decomp}
\end{align}
In the following, we control \eqref{eq:var-decomp} termwise. Firstly, the first term vanishes when $p\ge4$. For $2<p<4$, $V_{h,b}-\wh V_{h,b}=v_{h,b}v_{h,b}^\top\ind\{|v_{h,b}|>u\}$, whose largest singular value equals $|D_{h,b}|^2\ind\{|D_{h,b}|>u\}\le u^{2-p}|D_{h,b}|^p$; hence, by \eqref{eq:Dp-sum},
\begin{align}
    \IE\Bigg[\rho^{\star}\Big(\sum_{h,b} (\wh{V}_{h,b} - V_{h,b})\Big)\Bigg] \le u^{2-p} \sum_{h,b} \IE[|D_{h,b}|^p] \lesssim 2^{2kd/p} (k+1)^d. \label{eq:var-1}
\end{align}
On the other hand, for the second term in \eqref{eq:var-decomp}, Lemma \ref{lem:cov-2nd-term} instructs that
 \begin{align}
     \IE\Big[\rho^\star\big(\sum_{h,b} V_{h,b} - \Cov(Z_B)\big)\Big] \lesssim 2^{2kd/ r} (k+1)^{d+1 - 2/r} \nu_A, \ r=p\wedge 4. \label{eq:var-2}
 \end{align}
For the third term, \eqref{eq:ZB-conv} and the finiteness of $Q$ imply $\Cov(Z_B)\to\Cov(Z)$ as $B\to\infty$. We choose $B$ so large that
\begin{equation}\label{eq:var-3}
    \rho^\star\big(\Cov(Z_B) - \Cov(Z) \big)\le 2^{2kd/r}(k+1)^{d+1-2/r}.
\end{equation}
For the fourth term, Lemma \ref{lem:Haar-algebra}(b), Lemma \ref{lem:cov-compare} and \eqref{eq:residual-rate-main} yield
\begin{equation}\label{eq:var-4}
    \rho^\star\big(\HA(\Sigma_Q^{[2^k]}  - \Sigma_Q)\HA^\top\big)\le 2^{kd}\rho^\star\big(\Sigma_Q^{[2^k]}  - \Sigma_Q\big)\lesssim 2^{k(d-A)}=2^{2kd/r}2^{k(c-A)}\lesssim 2^{2kd/r}\nu_A,
\end{equation}
where we used $d=2d/r+c$ and $\nu_A\ge2^{k(c-A)}$. Combining \eqref{eq:var-1}-\eqref{eq:var-4}, and noting that $(k+1)^d\le(k+1)^{d+1-2/r}$ and $\nu_A\ge1$,
\begin{equation}\label{eq:Omega-final}
    \IE[\rho^{\star}(\wh{\Omega})]\lesssim 2^{2kd/ r} (k+1)^{d+1 - 2/r} \nu_A.
\end{equation}
\subsection{The martingale coupling, and the conclusion}\label{ssc:coupling}
We have now reached the point at which \cite{cattaneo2025yurinskii} can be directly invoked. For completeness, we state a version of the Corollary 2.2 therein that we end up applying.
\begin{theorem}[Corollary 2.2, Eqn. (5), of \cite{cattaneo2025yurinskii}]\label{thm:CMU}
On a complete probability space, let $\mathcal{H}_0\subseteq\mathcal{H}_1\subseteq\cdots\subseteq\mathcal{H}_T$ be countably generated $\sigma$-fields with $\mathcal{H}_0$ trivial, and let $X_1,\ldots,X_T$ be square-integrable $\R^{D}$-valued random vectors such that $X_t$ is $\mathcal{H}_t$-measurable and $\IE[X_t|\mathcal{H}_{t-1}]=0$. Let $$V_t=\IE[X_tX_t^\top|\mathcal{H}_{t-1}], \ S=\sum_{t=1}^TX_t\  \text{and } \Omega_T=\sum_{t=1}^TV_t-\Sigma,$$ where $\Sigma$ is a positive semi-definite matrix. Suppose that $\IE[X_{t,a}X_{t,b}X_{t,c}|\mathcal{H}_{t-1}]=0$ for all $t$ and all $a,b,c\in[D]$. Let $W_1,\ldots,W_T$ be i.i.d.\ $N(0,I_D)$, independent of $\mathcal{H}_T$, and
\[\beta_{\infty,3}=\sum_{t=1}^T\IE\Big[|X_t|^3|X_t|_\infty+|V_t^{1/2}W_t|^3|V_t^{1/2}W_t|_\infty\Big].\]
Then, for every $\bar\eta>0$, on an enlarged probability space there exists $T_G\sim N(0,\Sigma)$ such that
\begin{equation}\nonumber
    \IP(|S-T_G|_\infty>\bar\eta)\le 24\Big(\frac{\beta_{\infty,3}\,\phi_\infty(D)^3}{\bar\eta^4}\Big)^{1/4}+17\Big(\frac{\IE[\rho^\star(\Omega_T)]\,\phi_\infty(D)^2}{\bar\eta^2}\Big)^{1/3},
\end{equation}
where $\phi_\infty(D)=\sqrt{2\log(2D)}$.
\end{theorem}
To smoothen out the application of Theorem \ref{thm:CMU}, order the finitely many pairs $(h_j,b)$ and use the dictionary
\begin{equation}\nonumber
    \mathcal{H}_0=\{\emptyset,\Omega\},\  \mathcal{H}_t=\cG_{h_j,b},\   X_t=\wh D_{h_j,b},\   V_t=\wh V_{h_j,b},\  \Omega_T=\wh\Omega,\   D=2^{kd}.
\end{equation}
By Lemma \ref{lem:rademacher}, $\{X_t\}$ are martingale differences with $V_t=\IE[X_tX_t^\top|\mathcal{H}_{t-1}]$ and vanishing conditional third moments. Thus in view of $\phi_\infty(2^{kd})^2=2\log(2^{kd+1})\lesssim (k+1)$, \eqref{eq:beta-final} and \eqref{eq:Omega-final}, Theorem \ref{thm:CMU} instructs that with
\[\bar\eta:=\frac{\mathfrak{b}}{4(k+1)^d} \asymp x^{-3/2}2^{kd/r}\nu_A^{1/2}(k+1)^{d/2+1-1/r},\]
there exists $T_G\sim N(0,\Sigma)$ such that
\begin{align*}
    \IP(|\wh S-T_G|_\infty>\bar\eta)&\lesssim \Big(\frac{2^{4kd/r}(k+1)^{2d+3/2}}{x^{-6}2^{4kd/r}\nu_A^2(k+1)^{2d+4-4/r}}\Big)^{1/4}\\
    &\hspace{2cm}+\Big(\frac{2^{2kd/r}(k+1)^{d+2-2/r}\nu_A}{x^{-3}2^{2kd/r}\nu_A(k+1)^{d+2-2/r}}\Big)^{1/3}\\
    &\lesssim x^{3/2}(k+1)^{(4/r-5/2)/4}+x.
\end{align*}
Since $r>2$ implies $4/r-5/2<0$, and $x^{3/2}\le x$, we can choose the constant hiding in $\lesssim$ to ensure
\begin{equation}\label{eq:CMU-applied}
    \IP(|\wh S-T_G|_\infty>\bar\eta)\le x/4.
\end{equation}
Set $\bG^Q:=\HA^{-1}T_G$. Then $\bG^Q$ is a centered Gaussian vector with
\[\Cov(\bG^Q)=\HA^{-1}\Sigma\HA^{-\top}=\Sigma_Q=\Cov((X_i)_{i\in Q}).\]
Recalling $\HA\BX^{[2^k]}=Z$ and $Z_B=\sum_{h,b}D_{h,b}$, we can decompose $\BX^{[2^k]}-\bG^Q=\HA^{-1}(Z-Z_B)+\HA^{-1}\sum_{h,b}(D_{h,b}-\wh D_{h,b})+\HA^{-1}(\wh S-T_G)$. Therefore, the triangle inequality and \eqref{eq:haar-inverse} yield
\begin{align}
    \max_{m\in Q}\Big|\sum_{i\in[1,m]}(\tilde X^Q_i-G^Q_i)\Big|\le&\max_{m\in Q}\Big|\sum_{i\in[1,m]}(X_i-X^{[2^k]}_i)\Big|+(k+1)^d|Z-Z_B|_\infty\nonumber\\
    &+(k+1)^d\Big|\sum_{h,b}(D_{h,b}-\wh D_{h,b})\Big|_\infty+(k+1)^d|\wh S-T_G|_\infty.\label{eq:four-errors}
\end{align}
In (a), (b) (c) and (d) in the following, we show that each of the four terms on the right-hand side of \eqref{eq:four-errors} exceeds $\mathfrak{b}/4$ with probability at most of the order of $x$, so that we always have the option to choose suitable constants.

\subsection{Control of \eqref{eq:four-errors}}\label{ssc:4errors}
\begin{enumerate}[label=(\alph*)]
    \item By Markov's inequality, \eqref{eq:ga-m-dep} and \eqref{eq:mdep-vs-nu},
    \begin{align*}
    \IP\Big(\max_{m\in Q}\Big|\sum_{i\in[1,m]}(X_i-X^{[2^k]}_i)\Big|>\frac{\mathfrak{b}}4\Big)&\le\Big(\frac{4}{\mathfrak{b}}\Big)^p\Big\|\max_{m\in Q}\Big|\sum_{i\in[1,m]}(X_i-X^{[2^k]}_i)\Big|\Big\|_p^p\\
    &\lesssim\Big(\frac{2^{k(d/2-A)}}{\mathfrak{b}}\Big)^p\le x^{3p/2}\le x.
    \end{align*}
    \item In addition to \eqref{eq:var-3}, we choose $B$ so large that $\IP((k+1)^d|Z-Z_B|_\infty>\mathfrak{b}/4)\le x/4$; this is possible by \eqref{eq:ZB-conv}, since $k$, $x$ and $\mathfrak{b}$ are fixed. All other bounds above are uniform in $B$.
    \item If $p\ge4$, the third term vanishes. If $2<p<4$, then $r=p$, and Markov's inequality and \eqref{eq:trunc-approx} give
    \[\IP\Big((k+1)^d\Big|\sum_{h,b}(D_{h,b}-\wh D_{h,b})\Big|_\infty>\frac{\mathfrak{b}}4\Big)\lesssim\frac{(k+1)^{3d/2}2^{kd/p}}{ x^{-3/2}2^{kd/p}\nu_A^{1/2}(k+1)^{3d/2+1-1/p}}\lesssim x^{3/2}\le x.\]
    \item By \eqref{eq:CMU-applied}, $\IP((k+1)^d|\wh S-T_G|_\infty>\mathfrak{b}/4)=\IP(|\wh S-T_G|_\infty>\bar\eta)\le x$.
\end{enumerate}
 Since $Q=[2^k]^d$, the family $\{[1,m]:m\in Q\}$ coincides with $\{R:R\preceq Q,\ R\ne\emptyset\}$. Therefore, a union bound over (a)-(d) proves Theorem \ref{thm:GA}, in view of \eqref{eq:ga-m-dep}, \eqref{eq:ga-haar-appl}, \eqref{eq:trunc-approx}, \eqref{eq:beta-final}, \eqref{eq:Omega-final} and \eqref{eq:CMU-applied}.
\section{Proof of Theorem \ref{thm:spectral}}\label{sec:proof-spectral}
Throughout this section, we grant the assumptions of Theorem \ref{thm:main}. Recall $\gamma(h)=\Cov(X_0,X_h)$, and let $\mu$ denote the normalized Lebesgue measure on $\mathbb{T}^d:=[-\pi,\pi]^d$. To properly state our arguments, some notational niceties need to be taken care of; we write $\iu$ for the imaginary unit, $\overline{z}$ for the complex conjugate of $z\in\C$, and $\mathcal L(Z)$ for the law
of $Z$.

\subsection{The spectral density near the origin} \label{ssc:spec-origin}
By Lemma \ref{lem:cov-tail}, the Fourier series $f(\lambda)=\sum_{h\in\Z^d}\gamma(h)e^{-\iu\langle h,\lambda\rangle}$ converges absolutely and uniformly, so that $f$ is continuous; since $\gamma(-h)=\gamma(h)$ is real, $f$ is real and even, and $\gamma(h)=\int_{\mathbb{T}^d}e^{\iu\langle h,\lambda\rangle}f(\lambda)\mu(\mathrm{d}\lambda)$. Moreover, $f\ge0$: for $n\in\N$ and $\lambda\in\mathbb{T}^d$,
\[0\le\frac{1}{(2n+1)^d}\IE\bigg[\Big|\sum_{j\in\{-n,\ldots,n\}^d}X_je^{-\iu\langle j,\lambda\rangle}\Big|^2\bigg]=\sum_{h}\gamma(h)e^{-\iu\langle h,\lambda\rangle}\prod_{l=1}^d\Big(1-\frac{|h_l|}{2n+1}\Big)_+\longrightarrow f(\lambda),\]
by dominated convergence. Finally, $f(0)=\sum_h\gamma(h)=\sigma^2_\infty>0$.

\subsection{A common i.i.d random field}
Let $W$ be a centered complex Gaussian random measure on $\mathbb{T}^d$ with control measure $\mu$, so that, $\IE[W(B_1)\overline{W(B_2)}]=\mu(B_1\cap B_2)$, which is Hermitian, $W(-B)=\overline{W(B)}$. We will use the following properties of $W$:
\begin{itemize}
    \item the stochastic integral $\int\varphi\,\mathrm{d}W$ is real-valued whenever $\varphi(-\lambda)=\overline{\varphi(\lambda)}$;
    \item $\IE[\int\varphi\,\mathrm{d}W\,\overline{\int\psi\,\mathrm{d}W}]=\int\varphi\bar\psi\,\mathrm{d}\mu$; see \cite[Chapter I]{ibragimov1978gaussian}. 
\end{itemize}
Define, on the probability space $(\Omega^{\mathrm{sp}},\Fcal^{\mathrm{sp}},\IP^{\mathrm{sp}})$ the jointly Gaussian random variables
\[G^{\mathrm{sp}}_i:=\int_{\mathbb{T}^d}e^{\iu\langle i,\lambda\rangle}\sqrt{f(\lambda)}\,W(\mathrm{d}\lambda),\  Y^{\mathrm{sp}}_i:=\int_{\mathbb{T}^d}e^{\iu\langle i,\lambda\rangle}\,W(\mathrm{d}\lambda),\  i\in\Z^d,\]
with $\IE[G^{\mathrm{sp}}_iG^{\mathrm{sp}}_j]=\int e^{\iu\langle i-j,\lambda\rangle}f(\lambda)\mu(\mathrm{d}\lambda)=\gamma(i-j)$ and $\IE[Y^{\mathrm{sp}}_iY^{\mathrm{sp}}_j]=\int e^{\iu\langle i-j,\lambda\rangle}\mu(\mathrm{d}\lambda)=\ind\{i=j\}$. The difference $H_i:=G^{\mathrm{sp}}_i-\sigma_\infty Y^{\mathrm{sp}}_i=\int e^{\iu\langle i,\lambda\rangle}(\sqrt{f(\lambda)}-\sigma_\infty)W(\mathrm{d}\lambda)$ is a stationary Gaussian field with spectral density $f_H:=(\sqrt f-\sigma_\infty)^2$. Since $\sigma_\infty=\sqrt{f(0)}>0$,
\begin{equation}\label{eq:fH}
    f_H(\lambda)=\frac{(f(\lambda)-f(0))^2}{(\sqrt{f(\lambda)}+\sigma_\infty)^2}\le\sigma_\infty^{-2}(f(\lambda)-f(0))^2.
\end{equation}
Combining \eqref{eq:fH} with Lemma \ref{lem:spec-modulus}, we obtain $f_H(\lambda)\lesssim s^\varrho$, where
\[\varrho:=2A\ \text{ if }0<A<2,\ \varrho:=3\ \text{ if }A=2,\ \varrho:=4\ \text{ if }A>2,\]
since $s^4\log^2(e/s)\lesssim s^3$. In particular, $\varrho<1$ if and only if $A<1/2$, $\varrho=1$ if and only if $A=1/2$, and $\varrho\in(1,4]$ if $A>1/2$. Moreover, $s^\varrho\le|\lambda|_\infty^\varrho\le\sum_{j=1}^d|\lambda_j|^\varrho$.
Now we proceed towards the formal proof.
\begin{proof}[Proof of Theorem \ref{thm:spectral}]
Fix $k\ge0$, and let $Q$ be a translate of $[2^k]^d$. By stationarity, the non-empty lower-left sub-rectangles of $Q$ are translates of the prefixes $[1,m]$, $m\in[2^k]^d$, and the corresponding partial sums of $(H_i)$ have the same laws. By the spectral isometry, for $m\in[2^k]^d$,
\[\Var\Big(\sum_{i\in[1,m]}H_i\Big)=\int_{\mathbb{T}^d}f_H(\lambda)\prod_{j=1}^d|D_{m_j}(\lambda_j)|^2\mu(\mathrm{d}\lambda)\lesssim \sum_{j=1}^d\Big[\int_{-\pi}^\pi|u|^\varrho|D_{m_j}(u)|^2\frac{\mathrm{d}u}{2\pi}\Big]\prod_{l\ne j}m_l,\]
where we used Lemma \ref{eq:dirichlet} after recalling $D_m(\lambda)$ from therein. Consequently, uniformly over $m\in[2^k]^d$,
\begin{equation}\label{eq:prefix-var}
    \Var\Big(\sum_{i\in[1,m]}H_i\Big)\lesssim v_k^2,\  v_k^2:=\begin{cases}
        2^{k(d-2A)}, & 0<A<1/2,\\
        2^{k(d-1)}(k+1), & A=1/2,\\
        2^{k(d-1)}, & A>1/2;
    \end{cases}
\end{equation}
for instance, if $A<1/2$, then $m_j^{1-2A}\prod_{l\ne j}m_l\le 2^{k(d-2A)}$. Since there are $2^{kd}$ nonempty lower-left sub-rectangles of $Q$, the Gaussian tail bound and a union bound give, for $t>0$,
\[\IP^{\mathrm{sp}}\Big(\max_{R\preceq Q}\Big|\sum_{i\in R}H_i\Big|>t\Big)\le 2^{kd+1}\exp\Big(-\frac{t^2}{2\bar v_k^2}\Big), \  \bar v_k^2:=\max_{m\in[2^k]^d}\Var\Big(\sum_{i\in[1,m]}H_i\Big).\]
Choosing $t:=\{2(d+2)(k+1)\}^{1/2}\bar v_k$, the right-hand side is at most $ 2\cdot4^{-k}$. Finally, \eqref{eq:prefix-var} yields $t\lesssim w_A(k)$, with $w_A(k)$ as in \eqref{eq:wA}. This completes the proof.
\end{proof}

\section{\texorpdfstring{Auxiliary results}{Auxiliary results}}\label{sec:helper-results}
\subsection{Auxiliary results for Theorem \ref{thm:main}}\label{ssc:aux-main}

\begin{proposition}\label{prop:Q-blocks}
   Grant the assumptions of Lemma \ref{lem:maximal} for $(X_i)$, i.e., $\Delta_p<\infty$ for some $p>2$. Further assume that there exists a constant $A>0$ such that $ \Delta_p^{[m]}\lesssim (1+m)^{-A},\ m\ge1.$  Recall $\cQ$ from \eqref{eq:cq}, and for each $Q \in \cQ$, consider $\BX^Q$ as defined in \eqref{eq:Q-blocks}. Suppose $Q$ has sides of length $L$, where $L=2^k$ for some $k\in \N \cup \{0\}$. Then, with implied constants depending only on $d$, $p$, $A$ and $\Delta_p$,
    \[ \big\| \max_{R \preceq Q} |\sum_{i\in R}(X_i -  X_i^Q)|\big\|_p \lesssim \begin{cases}
        2^{k(d-1)/2},& A>1/2,\\
        2^{k(d-1)/2} (k+1), & A=1/2,\\
        2^{k(d/2-A)}, & 0 < A < 1/2.
    \end{cases} \]
\end{proposition}

\begin{lemma}[Vorob'ev--Berkes--Philipp coupling; Theorem 1.31 of \cite{dudley1999uniform}]\label{lem:gluing}
Let $E,F,H$ be Polish spaces with their Borel $\sigma$-fields. Let $\alpha$ be a probability measure on $E\times F$ and $\beta$ a probability measure on $F\times H$. If the $F$-marginals of $\alpha$ and $\beta$ agree, then there exists a probability measure $\tau$ on $E\times F\times H$ whose $(E,F)$-marginal is $\alpha$ and whose $(F,H)$-marginal is $\beta$.
\end{lemma}
See also \cite{vorobev1962consistent} and \cite{berkes1979approximation}. 

\begin{lemma}\label{lem:transfer}
Let $(\Omega_0,\Fcal_0,\IP_0)$ be a probability space, let $\mathbb{G}$ and $\mathbb{H}$ be standard Borel spaces, and let $\BX:\Omega_0\to\mathbb{G}$ be a random element with law $\mu$. Suppose $\tau$ is a probability measure on $\mathbb{G}\times\mathbb{H}$ with first marginal $\mu$. Then there is a probability measure $\tilde\IP$ on $(\tilde\Omega,\tilde\Fcal):=(\Omega_0\times\mathbb{H},\Fcal_0\otimes\mathcal{B}(\mathbb{H}))$ such that
\begin{enumerate}[label=\textup{(\alph*)}]
    \item the projection of $\tilde\IP$ onto $(\Omega_0,\Fcal_0)$ is exactly $\IP_0$;
    \item if $Z$ denotes the $\mathbb{H}$-coordinate, then $\Lcal_{\tilde\IP}(\BX,Z)=\tau$.
\end{enumerate}
\end{lemma}

\subsection{Auxiliary results for Theorem \ref{thm:GA}}\label{ssc:aux-GA}

\begin{lemma}\label{lem:M_h}
    Grant the assumptions of Theorem \ref{thm:GA}, and recall the notation of Appendix \ref{sec:proof-GA}. For $0\le s\le k$ and $h\in H_Q$, recall  $D_{h,b}^{(s)}=\zeta_{h,b}v^{(s)}_{h,b}$ from \eqref{eq:D_hb}, and denote $M_h^{(s)} := (\sum_{b=1}^B |D_{h,b}^{(s)}|^2)^{1/2}=(\sum_{b=1}^B |v_{h,b}^{(s)}|^2)^{1/2}$. Then, with $a_{s,h}:=\sum_{i\in Q}\delta^{(s)}_{i-h,p}\le\Theta_{s,p}$,
    \[ \| M_h^{(s)}\|_p \lesssim (k+1)^{d/2} a_{s,h}\lesssim (k+1)^{d/2}\Theta_{s,p},\]
    where the implied constants depend only on $p$; in particular, they do not depend on $B$, $h$ or $k$.
\end{lemma}
\begin{lemma} \label{lem:cov-2nd-term}
    Grant the assumptions of Theorem \ref{thm:GA}, and recall $D_{h,b}$, $V_{h,b}=D_{h,b}D_{h,b}^\top$ and $Z_B$ from Appendix \ref{sec:proof-GA}. Then $\Cov(Z_B)= \IE[\sum_{h,b} V_{h,b}]$, and moreover,
    \[ \IE\Big[\rho^\star\big(\sum_{h,b} V_{h,b} - \Cov(Z_B)\big)\Big] \lesssim 2^{2kd/ r} (k+1)^{d+1 - 2/r} \nu_A, \ r=p\wedge 4,\]
    where $\nu_A$ is as defined in Theorem \ref{thm:GA}, and the implied constant does not depend on $k$ or $B$.
\end{lemma}
\begin{lemma}\label{lem:matrix-moment}
Let $X_1,\dots,X_H$ be independent centered self-adjoint $D\times D$ random matrices. Let $q>1$, $q_\star=q\wedge 2$, and suppose
\begin{equation}\label{eq:matrix-q-bound}
  \max_{1\le j\le H}\|\rho^\star(X_j)\|_q \le K.
\end{equation}
Then
\begin{equation}\label{eq:matrix-moment-bound}
  \E\bigg[\rho^\star\!\Big(\sum_{j=1}^H X_j\Big)\bigg]
  \lesssim K H^{1/q_\star}[\log(2D)]^{1-1/q_\star}.
\end{equation}
\end{lemma}
\begin{lemma}\label{lem:cov-compare}
Suppose that $\Delta_p<\infty$ for some $p\ge2$. Then, for every finite $Q\subset\Z^d$ and $m\ge1$, with $\Sigma_Q:=\Cov((X_i)_{i\in Q})$ and $\Sigma_Q^{[m]}:=\Cov((X^{[m]}_i)_{i\in Q})$,
\begin{equation}\label{eq:cov-compare}
    \rho^\star\big(\Sigma_Q-\Sigma_Q^{[m]}\big)\lesssim \Delta_p^{[m]}\big(2\Delta_p+\Delta_p^{[m]}\big).
\end{equation}
\end{lemma}

\subsection{Auxiliary results for Theorem \ref{thm:spectral}}\label{ssc:aux-spectral}

\begin{lemma}[Covariance tail]\label{lem:cov-tail}
Suppose that $\Delta_p<\infty$ for some $p\ge2$, and let $\gamma(h)=\Cov(X_0,X_h)$. Then $\sum_{h\in\Z^d}|\gamma(h)|\le\Delta_2^2<\infty$. If, moreover, \eqref{eq:residual-rate-main} holds, then
\begin{equation}\label{eq:cov-tail}
    \sum_{|h|_\infty>m}|\gamma(h)|\lesssim (1+m)^{-A},\  m\ge1.
\end{equation}
\end{lemma}
\begin{lemma}\label{lem:spec-modulus}
Recall $f$ from \S\ref{ssc:spec-origin}. Let $s:=\min\{|\lambda|_\infty,1\}$. Then, for all $\lambda\in\mathbb{T}^d$,
\begin{equation}\label{eq:spec-modulus}
    |f(\lambda)-f(0)|\lesssim \begin{cases}
        s^A, & 0<A<2,\\
        s^2\log(e/s), & A=2,\\
        s^2, & A>2.
    \end{cases}
\end{equation}
\end{lemma}

\begin{lemma}\label{lem:dirichlet}
For $m\in\N$, let $D_m(u):=\sum_{v=1}^me^{\iu vu}$ denote the Dirichlet kernel. Then for every $\varrho\in(0,4]$ and $m\in\N$,
\begin{equation}\label{eq:dirichlet}
    \int_{-\pi}^\pi|u|^\varrho|D_m(u)|^2\frac{\mathrm{d}u}{2\pi}\lesssim \begin{cases}
        m^{1-\varrho}, & 0<\varrho<1,\\
        \log(em), & \varrho=1,\\
        1, & \varrho>1.
    \end{cases}
\end{equation}
\end{lemma}

\subsection{Additional auxiliary results}\label{ssc:aux-additional}
Finally, we collect the moment and maximal inequalities for the field, its $m$-dependent approximations and the surrogate fields, which are used throughout.

\begin{lemma}[Propositions 1 and 3 of \cite{el2013central}]\label{prop:coefficient}
For every finitely supported deterministic array $(a_i)$,
\begin{equation}\label{eq:coefficient-eps}
  \Big\|\sum_i a_iX_i\Big\|_p
  \lesssim \Delta_p\Big(\sum_i a_i^2\Big)^{1/2}.
\end{equation}
Moreover, for every $m\ge1$,
\begin{equation}\label{eq:coefficient-residual}
  \Big\|\sum_i a_i (X_i - X_i^{[m]})\Big\|_p
  \lesssim \Delta_p^{[m]}\Big(\sum_i a_i^2\Big)^{1/2}.
\end{equation}
\end{lemma}

Rectangular maximal inequalities for spatial data is not a recent topic by any means. For multiparameter
martingales adapted to product filtrations  see \citep{Cairoli1970} and
related settings
\citep{christofides1990maximal,hirsch1995potential,walsh2006martingales}.
For random fields, \citet{cuny2026weak} gives a Rosenthal inequality
under a causal innovation representation,
while \citet{bonnerjee2026fast} derives maximal bounds for linear and
finite-range spatial processes.
The lemma below follows from \citet{moricz1977moment}; see also
\citet{Moricz1979MultiparameterII,Moricz1983Rectangular}.
Combined with Lemma~\ref{prop:coefficient}, it supplies the rectangular
maximal bounds needed for the original field, its local approximation
residuals and the surrogate fields. Unlike \cite{bonnerjee2026fast} or \cite{cuny2026weak}, no separate maximal-inequality
assumption or restriction to coordinatewise nonnegative innovation
lags is needed here even for the wider class of spatial dependence. 
\color{black}

\begin{lemma}\label{lem:maximal}
For some $d\ge 1$, let $(Z_i)_{i\in\Z^d}$ be a centered random field such that, for some $p>2$ and $K<\infty$,
\begin{equation}\label{eq:maximal-hyp}
    \Big\|\sum_i a_iZ_i\Big\|_p\le K\Big(\sum_ia_i^2\Big)^{1/2}
\end{equation}
for every finitely supported deterministic array $(a_i)$. Then, for every $ Q=\prod_{j=1}^d\{a_j,\dots,a_j+L_j-1\},$ it follows,
\begin{equation}\label{eq:rect-maximal}
  \big \|
    \max_{R\preceq Q}
    |\sum_{i\in R}Z_i|\big\|_p
  \lesssim K \abs{Q}^{1/2}.
\end{equation}
In particular, in view of Lemma \ref{prop:coefficient}, if $\Delta_p<\infty$, then \eqref{eq:rect-maximal} applies to $(X_i)$ with some $K\lesssim\Delta_p$ and, for every $m\ge1$, to $(X_i-X_i^{[m]})$ with some $K\lesssim\Delta_p^{[m]}$. The same holds for the surrogate fields $(X^Q_i)$ and $(X^Q_i-X^{Q,[m]}_i)$ of \S\ref{ssc:blocking}, where $X^{Q,[m]}_i:=g_m(\eta^Q_{i-s};|s|_\infty\le m)$, since these fields have the same laws as $(X_i)$ and $(X_i-X_i^{[m]})$, respectively.
\end{lemma}
\section{Proofs of auxiliary results}\label{sec:helper-proofs}
\subsection{Proofs of results in Appendix \ref{sec:proof-GA}}

\begin{proof}[Proof of Lemma \ref{lem:Haar-algebra}]
By \eqref{eq:haar-values} and \eqref{eq:tensor-haar}, if $i\in Q_{\lambda}$, then $|\psi_{\lambda}(i)|= |Q_{\lambda}|^{-1/2}$, and therefore $\HA_{\lambda i}=\abs{Q_\lambda}^{1/2}\psi_\lambda(i)\in\{0,\pm1\},$
proving (a). Moving on, since $\HO$ is orthogonal,
\begin{equation*}
  \rho^\star(\HA)=\rho^\star(\HD)=\max_\lambda\abs{Q_\lambda}^{1/2}=2^{kd/2};
\end{equation*}
the maximum is attained by $\lambda=(\phi, \ldots, \phi)$, proving (b).

In one dimension, a site $i$ belongs to the father support and to exactly one wavelet support at each of the $k$ dyadic support scales, and therefore, exactly $k+1$ basis functions are nonzero at that site.  Tensorization gives $(k+1)^d$ nonzero entries in each column, each of magnitude $1$, proving (c).

Next, consider $[m]$ for some $m \le 2^k$. For a dyadic interval $I$ with $|I|\ge2$, $\langle b_I, \ind_{[m]}\rangle \ne 0$ only if $[m]$ cuts $I$, i.e., $[m]\cap I\ne\emptyset$ and $I\not\subseteq[m]$ (if $I\subseteq[m]$ the inner product vanishes since $b_I$ has mean zero). Therefore,
\[\abs{\inner{\ind_{[m]}, {b_I}}}
  =
  |I|^{-1/2} \min\{|I \cap [m]|,|I|-|I \cap [m]|\}\ind\{[m] \text{ cuts } I\}
  \le \frac{|I|^{1/2}}{2}\ind\{[m] \text{ cuts } I\},\]
while $|\langle\ind_{[m]},\phi\rangle|=m2^{-k/2}\le|I_\phi|^{1/2}$. Observe that at each of the $k$ wavelet scales, at most one \textit{dyadic} interval is cut by $[m]$. Consequently, $\sum_{b\in\cB_k}|\langle\ind_{[m]},b\rangle|/|I_b|^{1/2}\le k+1$, and for $m=(m_1, \ldots, m_d)\in Q$ the tensor sum factorizes:
\begin{align*}
\sum_{\lambda\in\cB_k^d}
  \frac{\abs{\inner{\ind_{[1,m]}, {\psi_\lambda}}}}
       {\abs{Q_\lambda}^{1/2}}
&=
\prod_{j=1}^d
\sum_{b\in\cB_k}
\frac{\abs{\inner{\ind_{[m_j]}, b}}}
     {\abs{I_b}^{1/2}}\le(k+1)^d,
\end{align*}
proving (d). Finally, $\HA^{-1}=\HO^{\top}\HD^{-1}$, so
\[ x_i=\sum_{\lambda\in\cB_k^d}
  \frac{\psi_{\lambda}(i)}{\abs{Q_\lambda}^{1/2}}(\HA x)_\lambda.\]
Therefore, for $m \in Q$, it follows,
\begin{align*}
\Big|\sum_{i\in[1,m]}(x_i-g_i)\Big|
&=\bigg|
  \sum_\lambda
  \frac{\inner{\ind_{[1,m]}, {\psi_\lambda}}}
       {\abs{Q_\lambda}^{1/2}}
  (\HA(x-g))_\lambda
\bigg|\\&\le
|\HA(x-g)|_\infty
\sum_\lambda
  \frac{\abs{\inner{\ind_{[1,m]}, {\psi_\lambda}}}}
       {\abs{Q_\lambda}^{1/2}}
\le (k+1)^d|\HA(x-g)|_\infty.
\end{align*}
This proves (e).
\end{proof}
\begin{proof}[Proof of Lemma \ref{lem:rademacher}]
(i) is Doob's martingale property. (ii) and (iii) follow from \eqref{eq:D_hb}, since $\zeta_{h,b}^2=1$ and $v_{h,b}$ is $\cG_{h,b-1}$-measurable. For (iv), $\zeta_{h,b}^3=\zeta_{h,b}$ and $\IE[\zeta_{h,b}|\cG_{h,b-1}]=0$. For the truncated version, by (ii) the indicator $\ind\{|D_{h,b}|\le u\}=\ind\{|v_{h,b}|\le u\}$ is $\cG_{h,b-1}$-measurable, so the truncated increment equals $\zeta_{h,b}v_{h,b}\ind\{|v_{h,b}|\le u\}$, which is again of the form \eqref{eq:D_hb}.
\end{proof}

\subsection{Proofs of results in \S\ref{ssc:aux-main}}

\begin{proof}[Proof of Proposition \ref{prop:Q-blocks}]
We use the finite-range approximations and a decomposition into layers according to distance from the block boundary. Recall $g_m$ from \eqref{eq:local-def}, and define $X_i^{Q, [m]}:= g_m(\eta_{i-s}^Q;|s|_\infty\le m)$. Let $Q=\prod_{j=1}^d\{a_j,\dots,a_j+L-1\}$, and for $i \in Q$, define its depth
    \[b_Q(i)=\min_{1\le j\le d}
  \min\{i_j-a_j,\ a_j+L-1-i_j\}.\]  Let $\cD_0=\{i\in Q:b_Q(i)=0\}$ and $
\cD_j=\{i\in Q:2^{j-1}\le b_Q(i)<2^j\},\ j\ge1$ be the \textit{peeled} layers of $Q$\color. Since $b_Q(i)\le (L-1)/2<2^{k-1}$, the number of non-empty layers is at most $k+1$. With $Q_t:=\{i\in Q: b_Q(i)\ge t\}$, which is a (possibly empty) concentric sub-box of $Q$, one has $\cD_0=Q\setminus Q_1$ and $\cD_j=Q_{2^{j-1}}\setminus Q_{2^j}$. The difference of two concentric boxes can be partitioned into at most $2d$ axis-aligned rectangles (slabs), by assigning a point to the first coordinate in which it lies outside the inner box. Therefore, each $\cD_j$ is a disjoint union of at most $2d$ axis-aligned rectangles $P_{j, l}$, each satisfying \begin{align}
    |P_{j, l}| \lesssim 2^j L^{d-1}. \label{eq:B_j_control}
\end{align}
Moreover, if $R\preceq Q$, then $R\cap P_{j,l}$ is either empty or a sub-rectangle anchored at the lower-left corner of $P_{j,l}$, i.e., $R\cap P_{j,l}\preceq P_{j,l}$: indeed, $R$ is anchored at the lower-left corner of $Q$, which is coordinatewise smaller than that of $P_{j,l}$. With this peeling, observe that triangle inequality dictates
\[ \big\| \max_{R \preceq Q} |\sum_{i\in R}(X_i -  X_i^Q)|\big\|_p \le \sum_{j,l}
  \big \|\max_{B\preceq P_{j,l}}
  |\sum_{i\in B}(X_i-X_i^Q)| \big\|_p. \]
When $j=0$, Lemma \ref{lem:maximal} implies that
\begin{align}
    \sum_{l} \Big(\big \|\max_{B\preceq P_{0,l}}
  |\sum_{i\in B}X_i| \big\|_p +\big \|\max_{B\preceq P_{0,l}}
  |\sum_{i\in B} X_i^Q| \big\|_p\Big) \lesssim \sum_{l} \Delta_p |P_{0, l}|^{1/2} \overset{(a)}{\lesssim} \Delta_p 2d L^{(d-1)/2}, \label{eq:j_0}
\end{align}
where $(a)$ uses \eqref{eq:B_j_control}. For $j \ge 1$, note that $i \in P_{j, l} \subset \cD_j$ implies that $b_Q(i) \ge 2^{j-1}$, which implies that for $|s|_{\infty}\le 2^{j-1}$ it holds $i-s\in Q$, so that $\eta_{i-s}=\eta_{i-s}^Q$ and therefore, $X_i^{[2^{j-1}]}= X_i^{Q, [2^{j-1}]}$ pathwise. Consequently, $X_i-X^Q_i=(X_i-X_i^{[2^{j-1}]})-(X^Q_i-X_i^{Q, [2^{j-1}]})$ on $P_{j,l}$, and Lemma \ref{lem:maximal}, applied to the two residual fields, yields
\begin{align}
   \sum_{l}  \big \|\max_{B\preceq P_{j,l}}
  |\sum_{i\in B}(X_i-X_i^Q)| \big\|_p &\le   \sum_{l} \Big(\big \|\max_{B\preceq P_{j,l}}
  |\sum_{i\in B}(X_i - X_i^{[2^{j-1}]})| \big\|_p \nonumber\\
  &\hspace{2cm}+\big \|\max_{B\preceq P_{j,l}}
  |\sum_{i\in B} (X_i^Q-X_i^{Q, [2^{j-1}]}) | \big\|_p\Big) \nonumber\\ &\lesssim \sum_{l} \Delta_p^{[2^{j-1}]} |P_{j, l}|^{1/2} \overset{(b)}{\lesssim} 2^{j(1/2-A)} L^{(d-1)/2}, \label{eq:j_general}
\end{align}
where $(b)$ employs \eqref{eq:residual-rate-main} and \eqref{eq:B_j_control}. Combining \eqref{eq:j_0}-\eqref{eq:j_general}, it follows that
\begin{align}
    \big\| \max_{R \preceq Q} |\sum_{i\in R}(X_i -  X_i^Q)|\big\|_p & \lesssim L^{(d-1)/2} \Big( 1+\sum_{j=1}^{k} 2^{j(1/2-A)} \Big) \nonumber\\
    &\lesssim \begin{cases}
        L^{(d-1)/2}, & A>1/2,\\
L^{(d-1)/2}(k+1), & A=1/2,\\
L^{d/2- A}, & A<1/2,
    \end{cases} \nonumber
\end{align}
where in the last case we used $\sum_{j=1}^k2^{j(1/2-A)}\lesssim 2^{k(1/2-A)}=L^{1/2-A}$. This completes the proof.
\end{proof}

\begin{proof}[Proof of Lemma \ref{lem:transfer}]
Since $\mathbb{G}$ and $\mathbb{H}$ are standard Borel, there is a probability kernel $K$ from $\mathbb{G}$ to $\mathbb{H}$ such that $\tau(\mathrm{d}x,\mathrm{d}z)=\mu(\mathrm{d}x)K(x,\mathrm{d}z)$ \citep[Chapter 6]{kallenberg2002foundations}. Since $\omega\mapsto K(\BX(\omega),B)$ is $\Fcal_0$-measurable for every $B\in\mathcal{B}(\mathbb{H})$, the formula
\[\tilde\IP(\mathrm{d}\omega,\mathrm{d}z):=\IP_0(\mathrm{d}\omega)\,K(\BX(\omega),\mathrm{d}z)\]
defines a probability measure on $(\tilde\Omega,\tilde\Fcal)$. Because $K(\BX(\omega),\mathbb{H})=1$, one has $\tilde\IP(G\times\mathbb{H})=\IP_0(G)$ for every $G\in\Fcal_0$, proving (a). If $\varphi:\mathbb{G}\times\mathbb{H}\to\R$ is bounded and Borel, then
\begin{align*}
\int\varphi(\BX(\omega),z)\,\tilde\IP(\mathrm{d}\omega,\mathrm{d}z)&=\int_{\Omega_0}\int_{\mathbb{H}}\varphi(\BX(\omega),z)K(\BX(\omega),\mathrm{d}z)\IP_0(\mathrm{d}\omega)\\
&=\int_{\mathbb{G}}\int_{\mathbb{H}}\varphi(x,z)K(x,\mathrm{d}z)\mu(\mathrm{d}x)=\int\varphi\,\mathrm{d}\tau,
\end{align*}
which proves (b).
\end{proof}

\subsection{Proofs of results in \S\ref{ssc:aux-GA}}

\begin{proof}[Proof of Lemma \ref{lem:M_h}]
    Fix $0\le s\le k$ and $h=h_j\in H_Q$. Since $\{D_{h_j,b}^{(s)}\}_{b\in [B]}$ is a martingale difference sequence in the Hilbert space $(\R^Q,|\cdot|)$ with respect to the filtration $\{\cG_{h_j,b}\}_{b \in [B]}$ (recall $\cG_{h_j,0}=\cG_{h_{j-1},B}$), Burkholder's square-function inequality for Hilbert-space valued martingales \citep[Theorem~3.1]{burkholder1988sharp}\color{black}{} entails that
    \begin{align}
         \|M_h^{(s)}\|_p \lesssim \Big\| \sum_{b=1}^B D_{h,b}^{(s)}\Big\|_p . \label{eq:burkholder}
    \end{align}
    Note that $\sum_b D_{h_j,b}^{(s)}= \IE[Z^{(s)} | \cG_{h_j, B}] - \IE[Z^{(s)} | \cG_{h_{j-1}, B}].$ Let $U'_{h_j}$ be a $\mathrm{Unif}(0,1)$ random variable independent of everything else, and let $Z^{(s), \{h_j\}}$ be the random vector defined similar to $Z^{(s)}$, but with $\eta_{h_j}$ replaced by $\eta_{h_j}':=F_\eta^{-1}(U'_{h_j})$. The vector $Z^{(s),\{h_j\}}$ is a function of $U'_{h_j}$ and $(U_u)_{u\ne h_j}$, which are jointly independent of $(B_{h_j,b})_{b\in[B]}$; moreover, $\cG_{h_{j-1},B}$ is generated by bits at sites other than $h_j$. Therefore,
    \[\IE[Z^{(s), \{h_j\}} | \cG_{h_j, B}]=\IE[Z^{(s), \{h_j\}} | \cG_{h_{j-1}, B}]=\IE[Z^{(s)} | \cG_{h_{j-1}, B}],\]
    where the last equality holds since replacing $U_{h_j}$ by $U'_{h_j}$ does not alter the joint law of $Z^{(s)}$ and the generators of $\cG_{h_{j-1},B}$. Consequently, by the conditional Jensen inequality,
    \[\Big\|\sum_bD^{(s)}_{h,b}\Big\|_p=\big\|\IE[Z^{(s)}-Z^{(s),\{h\}}|\cG_{h,B}]\big\|_p\le\|Z^{(s)}-Z^{(s),\{h\}}\|_p.\]
    Finally, $Z^{(s)}-Z^{(s),\{h\}}=\HA(\xi^{(s)}_i-\xi^{(s),\{h\}}_i)_{i\in Q}$, where $\xi^{(s),\{h\}}_i$ denotes $\xi^{(s)}_i$ with $\eta_h$ replaced by $\eta'_h$. By stationarity, replacing $\eta_h$ in $\xi^{(s)}_i$ is equivalent to replacing $\eta_0$ in $\xi^{(s)}_{i-h}$, so that $\|\xi^{(s)}_i-\xi^{(s),\{h\}}_i\|_p=\delta^{(s)}_{i-h,p}$. Hence, \eqref{eq:burkholder} along with \eqref{eq:haar-column} entail that
    \begin{align}
        \|M_h^{(s)}\|_p \lesssim \| Z^{(s)} - Z^{(s), \{h\}}\|_p \lesssim (k+1)^{d/2} \sum_{i\in Q} \delta_{i-h, p}^{(s)} = (k+1)^{d/2}a_{s,h}\lesssim (k+1)^{d/2}\Theta_{s,p}. \nonumber
    \end{align}
    This completes the proof.
\end{proof}
\begin{proof}[Proof of Lemma \ref{lem:cov-2nd-term}]
By Lemma \ref{lem:rademacher}, the increments $D_{h,b}$ are orthogonal martingale differences with $\IE[D_{h,b}D_{h,b}^\top]=\IE[V_{h,b}]$, and $Z_B=\sum_{h,b}D_{h,b}$ is centered. Therefore $\Cov(Z_B)=\IE[\sum_{h,b}V_{h,b}]$. For $0\le t\le s\le k$ and $h\in H_Q$, define $T_h^{s,s}=\sum_b v_{h,b}^{(s)}v_{h,b}^{(s)\top}$ and, for $t<s$,
\[T_h^{s,t}=\sum_b \big(v_{h,b}^{(s)}v_{h,b}^{(t)\top}+ v_{h,b}^{(t)}v_{h,b}^{(s)\top} \big).\]
Since $V_{h,b}=v_{h,b}v_{h,b}^\top$ with $v_{h,b}=\sum_{s=0}^kv^{(s)}_{h,b}$, we can write
\[\sum_{h,b} V_{h,b} - \Cov(Z_B) = \sum_{0\le t\le s\le k} \sum_h (T_h^{s,t} - \IE[T_h^{s,t}]).\]

\textit{Moment bound.} Since $\rho^\star(xy^\top)=|x||y|$, the Cauchy-Schwarz inequality in $b$ yields $\rho^\star(T^{s,t}_h)\le 2\sum_b|v^{(s)}_{h,b}||v^{(t)}_{h,b}|\le 2M^{(s)}_hM^{(t)}_h$. Hence, by H\"older's inequality and Lemma \ref{lem:M_h},
\begin{align}
    \|\rho^{\star}(T_h^{s,t})\|_{p/2} \le 2 \|M^{(s)}_h\|_p\|M^{(t)}_h\|_p\lesssim (k+1)^d \Theta_{s,p}\Theta_{t,p}, \nonumber
\end{align}
which readily implies that
\begin{align}
    \|\rho^{\star}(T_h^{s,t}- \IE[T_h^{s,t}])\|_{p/2}\lesssim (k+1)^d \Theta_{s,p}\Theta_{t,p}. \label{eq:indiv-T_h-bound}
\end{align}

\textit{Locality.} We claim that, for $t\le s$, $T_h^{s,t}$ is measurable with respect to $\mathcal{L}_{h,s}:=\sigma(B_{u,c}: |u-h|_{\infty} \le 2^{s+1}, c\in [B])$. Indeed, by \eqref{eq:D_hb},
\[v^{(s)}_{h,b}=\sum_{i\in Q}\HA_{\cdot i}\Big(\IE[\xi^{(s)}_i\ind\{B_{h,b}=1\}|\cG_{h,b-1}]-\IE[\xi^{(s)}_i\ind\{B_{h,b}=0\}|\cG_{h,b-1}]\Big).\]
If $|i-h|_\infty>2^s$, then by \eqref{eq:xi-local} $\xi^{(s)}_i$ is a function of $(U_u)_{u\ne h}$, so that $B_{h,b}$ is independent of $\sigma(\xi^{(s)}_i)\vee\cG_{h,b-1}$; consequently, both conditional expectations equal $\frac12\IE[\xi^{(s)}_i|\cG_{h,b-1}]$, and the $i$-th summand vanishes. If $|i-h|_\infty\le 2^s$, then $\xi^{(s)}_i\ind\{B_{h,b}=\eps\}$ is a function of $(U_u)_{|u-h|_\infty\le 2^{s+1}}$, which is independent of the bits in $\cG_{h,b-1}$ at sites $u$ with $|u-h|_\infty>2^{s+1}$. Hence, the conditional expectations given $\cG_{h,b-1}$ coincide with those given the sub-$\sigma$-field generated by the revealed bits at sites $u$ with $|u-h|_\infty\le 2^{s+1}$, which is contained in $\mathcal{L}_{h,s}$. Thus $v^{(s)}_{h,b}$, and likewise $v^{(t)}_{h,b}$ for $t\le s$, are $\mathcal{L}_{h,s}$-measurable, proving the claim.

\textit{Spatial coloring.} Choose $m_s = 2^{s+2}+1$, and for $w \in \{0,\ldots,m_s-1\}^d,$ let $\mathcal{W}_w:=\{h\in H_Q: h \bmod m_s = w \}$, where $\bmod$ operates coordinate-wise. If $h_1\ne h_2\in\mathcal{W}_w$, then $|h_1-h_2|_\infty\ge m_s>2\cdot2^{s+1}$, so that $\{T_h^{s,t}-\IE[T^{s,t}_h]\}_{h\in\mathcal{W}_w}$ are independent, centered, symmetric $2^{kd}\times 2^{kd}$ matrices. Moreover, $|\mathcal{W}_w|\le\lceil 3\cdot 2^k/m_s\rceil^d\lesssim 2^{(k-s)d}$. Writing
\[ \sum_h (T_h^{s,t} - \IE[T_h^{s,t}]) = \sum_w \sum_{h\in \mathcal{W}_w}(T_h^{s,t} - \IE[T_h^{s,t}]), \]
Lemma \ref{lem:matrix-moment} along with \eqref{eq:indiv-T_h-bound}, provides
\begin{align}
 \IE\bigg[\rho^\star\Big(\sum_{h\in \mathcal{W}_w}(T_h^{s,t} - \IE[T_h^{s,t}])\Big)\bigg] \lesssim (k+1)^{d+1-2/r} \Theta_{s,p}\Theta_{t,p} 2^{2(k-s)d/r} \nonumber.
\end{align}
This readily implies, in light of $m_s^d \lesssim 2^{sd}$ and \eqref{eq:bound-on-theta}, that
\begin{align}
    \IE\Big[\rho^\star\big(\sum_h (T_h^{s,t} - \IE[T_h^{s,t}])\big)\Big] &\lesssim 2^{2kd/r}(k+1)^{d+1-2/r}  2^{s d(1-2/r)}\Theta_{s,p}\Theta_{t,p}\nonumber\\
    &\lesssim 2^{2kd/r}(k+1)^{d+1-2/r}2^{s(d(1-2/r)-A)}2^{-tA}. \nonumber
\end{align}
Summing over $0\le t\le s\le k$, and using $\sum_{t\ge0}2^{-tA}<\infty$ and $\sum_{s=0}^k2^{s(d(1-2/r)-A)}=\nu_A$, completes the proof.
\end{proof}
\begin{proof}[Proof of Lemma \ref{lem:matrix-moment}]
Since $q_\star\le q$, \eqref{eq:matrix-q-bound} implies
$\E[\rho^\star(X_j)^{q_\star}]\le K^{q_\star}$ for every $j$. Write $L=\log(2D)$. If $H<L$, then the triangle inequality gives
$\E[\rho^\star(\sum_{j=1}^H X_j)]\le HK\le K H^{1/q_\star}L^{1-1/q_\star}$.
Hence assume $H\ge L$ and set
\begin{equation}\label{eq:matrix-truncation-level}
  \tau=K(H/L)^{1/q_\star}.
\end{equation}

Define 
$Y_j=X_j\mathbf{1}_{\{\rho^\star(X_j)\le\tau\}}-\E[X_j\mathbf{1}_{\{\rho^\star(X_j)\le\tau\}}]$. Then $Y_1,\dots,Y_H$ are independent, centered, self-adjoint, and $\rho^\star(Y_j)\le 2\tau$. Since $\E X_j=0$,
\[
X_j-Y_j
=
X_j\mathbf{1}_{\{\rho^\star(X_j)>\tau\}}
-
\E\!\big[
X_j\mathbf{1}_{\{\rho^\star(X_j)>\tau\}}
\big].
\]
Therefore, using $x\mathbf{1}_{\{x>\tau\}}\le x^{q_\star}\tau^{1-q_\star}$,
\begin{equation}\label{eq:matrix-discard}
  \E\bigg[\rho^\star\!\Big(\sum_{j=1}^H(X_j-Y_j)\Big)\bigg]
  \le 2HK^{q_\star}\tau^{1-q_\star}
  \lesssim K H^{1/q_\star}L^{1-1/q_\star}.
\end{equation}
Moreover,
$\rho^\star(\sum_{j=1}^H \E Y_j^2)
\le \sum_{j=1}^H \E[\rho^\star(A_j)^2]
\le HK^{q_\star}\tau^{2-q_\star}$,
where we used $\E Y_j^2\preceq\E A_j^2$ and
$x^2\mathbf{1}_{\{x\le\tau\}}\le \tau^{2-q_\star}x^{q_\star}$.
Since $\rho^\star(Y_j)\le 2\tau$, Theorem~6.1 of
\cite{tropp2012user}, applied to $\sum_{j=1}^H Y_j$ and
$-\sum_{j=1}^H Y_j$, gives, for every $t\ge 0$,
\[
\IP\Bigg(
\rho^\star\!\Big(\sum_{j=1}^H Y_j\Big)\ge t
\Bigg)
\le
2D\exp\Bigg(
-\frac{t^2}
{2\{ HK^{q_\star}\tau^{2-q_\star}+2\tau t/3\}}
\Bigg).
\]
Thus, since $L=\log(2D)$, for every $u\ge 0$,
\[
\IP\Bigg(
\rho^\star\!\Big(\sum_{j=1}^H Y_j\Big)
\ge
\sqrt{2 HK^{q_\star}\tau^{2-q_\star}(L+u)}
+\frac{4\tau}{3}(L+u)
\Bigg)
\le e^{-u}.
\]
Integrating this bound over $u\ge0$ gives
\[
\E\bigg[\rho^\star\!\Big(\sum_{j=1}^H Y_j\Big)\bigg]
\lesssim
\sqrt{HK^{q_\star}\tau^{2-q_\star}L}
+\tau L \lesssim K H^{1/q_\star}L^{1-1/q_\star}.
\]
This, together with
\eqref{eq:matrix-discard}, proves \eqref{eq:matrix-moment-bound}.
\end{proof}
\begin{proof}[Proof of Lemma \ref{lem:cov-compare}]
Fix $a\in\R^Q$ with $|a|=1$, and let $Y:=\sum_{i\in Q}a_iX_i$ and $Y':=\sum_{i\in Q}a_iX^{[m]}_i$. Then
\[|a^\top(\Sigma_Q-\Sigma_Q^{[m]})a|=|\Var(Y)-\Var(Y')|=|\IE[(Y-Y')(Y+Y')]|\le\|Y-Y'\|\,\|Y+Y'\|.\]
By \eqref{eq:coefficient-residual}, $\|Y-Y'\|\le\|Y-Y'\|_p\lesssim \Delta_p^{[m]}$, and by \eqref{eq:coefficient-eps}, $\|Y\|\lesssim \Delta_p$; hence $\|Y+Y'\|\le 2\|Y\|+\|Y-Y'\|\lesssim (2\Delta_p+\Delta_p^{[m]})$. Taking the supremum over $a$ proves \eqref{eq:cov-compare}, since $\Sigma_Q-\Sigma^{[m]}_Q$ is symmetric.
\end{proof}

\subsection{Proofs of results in \S\ref{ssc:aux-spectral}}

\begin{proof}[Proof of Lemma \ref{lem:cov-tail}]
Fix a bijection $\vartheta:\N\to\Z^d$, let $\mathcal{E}_\ell:=\sigma(\eta_{\vartheta(1)},\ldots,\eta_{\vartheta(\ell)})$ with $\mathcal{E}_0$ trivial, and let $P_\ell Y:=\IE[Y|\mathcal{E}_\ell]-\IE[Y|\mathcal{E}_{\ell-1}]$. Since $X_j\in L^2$ is centered and measurable with respect to $\sigma(\cup_\ell\mathcal{E}_\ell)$, martingale convergence gives $X_j=\sum_{\ell\ge1}P_\ell X_j$ in $L^2$. We first show that
\begin{equation}\label{eq:projection-bound}
    \|P_\ell X_j\|\le\delta_{j-\vartheta(\ell),2}.
\end{equation}
Let $u=\vartheta(\ell)$, and let $X_j^{(u)}$ be obtained by replacing $\eta_u$ by an independent copy $\eta'_u$. Since $X_j^{(u)}$ is a function of $\eta'_u$ and $(\eta_v)_{v\ne u}$, and $\eta_u$ is independent of these, $\IE[X^{(u)}_j|\mathcal{E}_\ell]=\IE[X^{(u)}_j|\mathcal{E}_{\ell-1}]=\IE[X_j|\mathcal{E}_{\ell-1}]$, the last equality holding since replacing $\eta_u$ does not alter the joint law of $X_j$ and the generators of $\mathcal{E}_{\ell-1}$. Therefore $P_\ell X_j=\IE[X_j-X^{(u)}_j|\mathcal{E}_\ell]$, and the conditional Jensen inequality and stationarity yield \eqref{eq:projection-bound}. Orthogonality of the projections and the Cauchy-Schwarz inequality now give
\[|\gamma(h)|=\Big|\sum_{\ell\ge1}\IE[(P_\ell X_0)(P_\ell X_h)]\Big|\le\sum_{u\in\Z^d}\delta_{-u,2}\,\delta_{h-u,2},\]
and summing over $h$ yields $\sum_h|\gamma(h)|\le\Delta_2^2\le\Delta_p^2$. Next, if $|i|_\infty>m$, then $X_i^{[m]}=g_m(\eta_{i-s};|s|_\infty\le m)$ does not involve $\eta_0$, so that $(X_i-X_i^{[m]})-(X_i-X_i^{[m]})^\star=X_i-X_i^\star$; hence $\delta_{i,2}=\delta^{[m]}_{i,2}\le\delta^{[m]}_{i,p}$ for $|i|_\infty>m$, and
\begin{equation}\label{eq:tail-delta}
    \Theta_2(m):=\sum_{|i|_\infty>m}\delta_{i,2}\le\Delta_p^{[m]}\lesssim (1+m)^{-A}.
\end{equation}
Finally, if $|h|_\infty>m$, then for every $u$ either $|u|_\infty>m/2$ or $|h-u|_\infty>m/2$. Therefore,
\[\sum_{|h|_\infty>m}|\gamma(h)|\le\sum_{|u|_\infty>m/2}\delta_{-u,2}\sum_h\delta_{h-u,2}+\sum_u\delta_{-u,2}\sum_{|h-u|_\infty>m/2}\delta_{h-u,2}\le 2\Delta_2\,\Theta_2(\lfloor m/2\rfloor),\]
which, in view of \eqref{eq:tail-delta}, proves \eqref{eq:cov-tail} for $m\ge2$; the case $m=1$ follows from $\sum_h|\gamma(h)|<\infty$.
\end{proof}
\begin{proof}[Proof of Lemma \ref{lem:spec-modulus}]
The case $\lambda=0$ is immediate.\color{black}{} If $|\lambda|_\infty\ge1$, then $s=1$ and $|f(\lambda)-f(0)|\le2\sum_h|\gamma(h)|\lesssim 1$. Let $0<|\lambda|_\infty<1$, so that $s=|\lambda|_\infty$. Since $f(\lambda)-f(0)=\sum_h\gamma(h)(\cos\langle h,\lambda\rangle-1)$ and $|1-\cos\langle h,\lambda\rangle|\le\min\{2,d^2|h|_\infty^2s^2/2\}$,
\[|f(\lambda)-f(0)|\lesssim s^2\sum_{0<|h|_\infty\le 1/s}|h|_\infty^2|\gamma(h)|+2\sum_{|h|_\infty>1/s}|\gamma(h)|.\]
By \eqref{eq:cov-tail}, $\sum_{|h|_\infty>1/s}|\gamma(h)|\lesssim(1+\lfloor1/s\rfloor)^{-A}\lesssim s^A$. For the first sum, separate the shell $|h|_\infty=1$ and cover the remaining terms by the dyadic annuli $2^{j-1}<|h|_\infty\le2^j$, $1\le j\le J:=\lceil\log_2(1/s)\rceil$. The first shell contributes $\sum_{|h|_\infty=1}|h|_\infty^2|\gamma(h)|\le\sum_h|\gamma(h)|\lesssim 1$, and \eqref{eq:cov-tail} bounds the contribution of the $j$-th annulus by $2^{2j}\sum_{|h|_\infty>2^{j-1}}|\gamma(h)|\lesssim 2^{j(2-A)}$. Consequently,
\[
\sum_{0<|h|_\infty\le1/s}|h|_\infty^2|\gamma(h)|
\lesssim \Big(1+\sum_{j=1}^{J}2^{j(2-A)}\Big) \lesssim \begin{cases}
    s^{A-2}, & \text{if } A<2,\\
    \log(e/s), & \text{if } A=2,\\
    1,& \text{if } A>2.
\end{cases}.
\]
Noting that $s^A\le s^2$ for $A>2$ proves \eqref{eq:spec-modulus}.
\end{proof}
\begin{proof}[Proof of Lemma \ref{lem:dirichlet}]
Since $|\sin(u/2)|\ge|u|/\pi$ for $|u|\le\pi$, we have $|D_m(u)|=|\sin(mu/2)/\sin(u/2)|\le\min\{m,\pi/|u|\}$. Therefore, by symmetry,
\[
\int_{-\pi}^\pi|u|^\varrho|D_m(u)|^2\frac{\mathrm{d}u}{2\pi}\lesssim m^2\int_0^{1/m}u^\varrho\,\mathrm{d}u+\int_{1/m}^\pi u^{\varrho-2}\,\mathrm{d}u\lesssim m^{1-\varrho}+\begin{cases}
    m^{1-\varrho}, & 0<\varrho<1,\\
    \log(em), & \varrho=1,\\
    1, & \varrho>1.
\end{cases}
\]
This proves \eqref{eq:dirichlet}.
\end{proof}

\subsection{Proofs of results in \S\ref{ssc:aux-additional}}

\begin{proof}[Proof of Lemma \ref{lem:maximal}]
Fix a rectangle $R'\subseteq Q$. Applying \eqref{eq:maximal-hyp} with $a_i=\ind\{i\in R'\}$ yields $\IE[|\sum_{i\in R'}Z_i|^p]\le (K^2|R'|)^{p/2}$. Since $R'\mapsto K^2|R'|$ is additive (in particular, superadditive) over disjoint rectangles, and $p/2>1$, the claim \eqref{eq:rect-maximal} follows directly from Theorem 7 of \cite{moricz1977moment}. 
\end{proof}
\end{document}